\documentclass[12pt]{article}
\usepackage{amsmath,amssymb,amsthm}
\usepackage{mathrsfs,bbm,calligra}

\usepackage{mathtools} %\coloneqとか\dcasesとかshowonlyrefsとか
\mathtoolsset{showonlyrefs=true}

\usepackage{enumitem}
\setlist[enumerate]{label=(\arabic*)}

\usepackage{tikz}
\usepackage{dynkin-diagrams}
\usetikzlibrary{cd,arrows.meta,calc,intersections,backgrounds}

\numberwithin{equation}{section} %数式番号を節ごとにする
\theoremstyle{plain}
\newtheorem{theorem}[equation]{Theorem} %定理番号と数式番号を通しにする
\newtheorem{proposition}[equation]{Proposition}
\newtheorem{lemma}[equation]{Lemma}
\newtheorem{corollary}[equation]{Corollary}
\newtheorem{conjecture}[equation]{Conjecture}

\theoremstyle{definition}
\newtheorem{definition}[equation]{Definition}
\newtheorem{example}[equation]{Example}

\theoremstyle{remark}
\newtheorem{remark}[equation]{Remark}

\usepackage{zref-clever} %賢い参照
\zcsetup{lang=english,cap=true}
\zcRefTypeSetup{conjecture}{
    Name-sg=Conjecture,
    Name-pl=Conjectures,
    name-sg=conjecture,
    name-pl=conjectures
}
\zcRefTypeSetup{equation}{
    Name-sg=,
    Name-pl=,
    name-sg=,
    name-pl=,
    refbounds={,(,),}
}
\AddToHook{env/theorem/begin}{
    \zcsetup{countertype={equation=theorem}}
}
\AddToHook{env/definition/begin}{
    \zcsetup{countertype={equation=definition}}
}
\AddToHook{env/lemma/begin}{
    \zcsetup{countertype={equation=lemma}}
}
\AddToHook{env/corollary/begin}{
    \zcsetup{countertype={equation=corollary}}
}
\AddToHook{env/proposition/begin}{
    \zcsetup{countertype={equation=proposition}}
}
\AddToHook{env/example/begin}{
    \zcsetup{countertype={equation=example}}
}
\AddToHook{env/remark/begin}{
    \zcsetup{countertype={equation=remark}}
}
\AddToHook{env/conjecture/begin}{
    \zcsetup{countertype={equation=conjecture}}
}

\usepackage[hidelinks]{hyperref}

\newcommand{\Z}{\mathbb{Z}}
\newcommand{\Q}{\mathbb{Q}}
\newcommand{\R}{\mathbb{R}}
\newcommand{\C}{\mathbb{C}}

\newcommand{\GL}{\mathrm{GL}}
\newcommand{\SL}{\mathrm{SL}}

\DeclareMathOperator{\Aut}{Aut}
\DeclareMathOperator{\Ker}{Ker}

\newcommand{\iHom}{\operatorname{\mathscr{H}\text{\kern -4pt\textcalligra{\large om}}}}

\DeclarePairedDelimiter{\abs}{\lvert}{\rvert}

\newcommand{\Ldual}[1]{{}^{\mathrm{L}}{#1}}
\newcommand{\Fr}{\mathrm{Frob}}
\newcommand{\ur}{\mathrm{ur}}
\DeclareMathOperator{\Lie}{Lie}
\newcommand{\cusp}{\mathrm{cusp}}
\newcommand{\disc}{\mathrm{disc}}
\DeclareMathOperator{\Irr}{Irr}
\DeclareMathOperator{\Ad}{Ad}
\DeclareMathOperator{\Out}{Out}
\newcommand{\unip}{\mathrm{unip}}
\newcommand{\LLC}{\mathrm{LLC}}
\DeclareMathOperator{\cind}{c-Ind}
\newcommand{\scon}{\mathrm{sc}}
\newcommand{\der}{\mathrm{der}}
\newcommand{\aff}{\mathrm{aff}}
\newcommand{\PGL}{\mathrm{PGL}}
\DeclareMathOperator{\Stab}{Stab}
\newcommand{\PU}{\mathrm{PU}}
\newcommand{\Sp}{\mathrm{Sp}}
\newcommand{\SO}{\mathrm{SO}}

\DeclareMathOperator{\fdeg}{fdeg}
\DeclareMathOperator{\vol}{vol}

\title{Parametrization of
 supercuspidal representations of depth zero for some simple adjoint groups}
\date{}
\author{Amoru Fujii}
\begin{document}
\maketitle
\begin{abstract}
    We construct a surjective map from the set of conjugacy classes of depth-zero cuspidal enhanced L-parameters to that of isomorphism classes of depth-zero supercuspidal representations for simple adjoint groups, and check the bijectivity in various cases. 
    We also prove that the Hiraga--Ichino--Ikeda conjecture on the formal degree of essentially square-integrable irreducible representations holds for this parametrization if it is bijective.
\end{abstract}

\section{Introduction}

Let $F$ be a non-Archimedean local field and $G$ a connected reductive group over $F$. It is expected that there exists a suitable parametrization, called the \emph{local Langlands correspondence}, of irreducible representations of $G(F)$ with \emph{L-parameters} for $G$, which is an analog of Galois representations. 

As we tackle this conjecture, it is effective to restrict the set of irreducible representations to a certain class; \emph{depth-zero supercuspidal representations}. In fact, supercuspidal representations are building blocks of all irreducible representations as in \cite[{VI.7.1.1}]{Renard}, and \cite[Theorem 15.1]{Yu2001} gives a way to construct almost all supercuspidal representations of positive depth from those of depth zero. On the correspondence for representations in this class, we have two important results; one is \cite[Theorem 4.5.3]{DR} for \emph{regular} supercuspidal representations, and the other is \cite[Theorem 2]{FOS} for \emph{unipotent} supercuspidal representations.

A remarkable feature of depth-zero irreducible representations is that they are related with irreducible representations of finite groups of Lie type. According to the Deligne--Lusztig theory, these representations correspond to conjugacy classes of elements of the dual group. In this viewpoint, regular supercuspidal representations correspond to semisimple elements, and unipotent supercuspidal representations correspond to unipotent elements. Then it is expected to extend these results for depth-zero supercuspidal representations which correspond to elements neither semisimple nor unipotent. In this paper, we treat this problem and have the following result:

\begin{theorem}[{Theorems~\ref{thm:general_LLC}~and~\ref{thm:LLC_simple_adjoint}}]\label{thm:main1}
    Let $G$ be a simple adjoint group over $F$ which splits over an unramified extension, and $\Phi_{\mathrm{e}}(G)_{0,\cusp}$ the set of depth-zero discrete L-parameters for $G$ with a cuspidal enhancement. Also we define $\Irr(G(F))_{0,\cusp}$ as the set of isomorphism classes of supercuspidal irreducible representations of depth zero. Then there exists a surjective map
    \[
        \LLC\colon {\Phi_{\mathrm{e}}(G)_{0,\cusp}/\mathord{\sim}}\twoheadrightarrow \Irr(G(F))_{0,\cusp},
    \]
    where $\sim$ is the conjugacy equivalence under $\widehat{G}$. Moreover, it is bijective in the following cases:
    \begin{enumerate}[label={$(\arabic*)$}]
        \item $G$ is of type $A_n,E_6,E_8,F_4$ or $G_2$.
        \item $G$ is an inner form of ${}^3 D_4$.
        \item $G$ is split of type $C_n$ or $E_7$.
        \item $G$ is of type $B_n$ or quasi-split of type ${}^2 D_{2n}$ or $D_{2n+1}$, if the residual characteristic $p$ of $F$ is odd.
    \end{enumerate}
\end{theorem}

We need to emphasize that the map $\LLC$ above depends on several non-canonical choices (Remark~\ref{rem:non_can_choices}). 
For $(\varphi,\epsilon)\in \Phi_\mathrm{e}(G)_{0,\cusp}$ with $\varphi|_{\SL_2(\C)}$ trivial, however, $\LLC(\varphi,\epsilon)$ is defined without such ambiguity and coincides with the irreducible representation which corresponds to $(\varphi,\epsilon)$ under the correspondence constructed in \cite[Theorem 4.5.3]{DR} (see Remark~\ref{rem:compari_to_DR}). Also, $\LLC$ is by construction an extension of the correspondence for unipotent supercuspidal representations obtained in \cite[Theorem 2]{FOS}, which similarly has some ambiguity.
We also have a result on the formal degree of depth-zero supercuspidal representations. In \cite[Conjecture 1.4]{HII}, Hiraga, Ichino and Ikeda present a conjecture (called the \emph{Hiraga--Ichino--Ikeda conjecture}) that calculates the formal degree of an essentially square-integrable representation in terms of its enhanced L-parameter. In \cite[p.~295]{HII}, they also prove that the correspondence constructed by \cite[Theorem 4.5.3]{DR} satisfies this conjectural equality. On the other hand, \cite[Theorem 3]{FOS} shows that the Hiraga--Ichino--Ikeda conjecture holds for their parametrization of unipotent supercuspidal representations. We also extend these results:

\begin{theorem}[{Theorem~\ref{thm:HIIconj}}]\label{thm:main2}
    Suppose that the map $\LLC$ in Theorem~\ref{thm:main1} is bijective. Then for any $(\varphi,\epsilon)\in \Phi_\mathrm{e}(G)_{0,\cusp}$, 
    \[
        \fdeg(\pi_{\varphi,\epsilon})=\frac{\dim(\epsilon)}{\abs{\mathcal{S}_\varphi}}\cdot \abs{\gamma(0,\varphi,\Ad,\psi)},
    \]
    where $\pi_{\varphi,\epsilon}=\LLC(\varphi,\epsilon)$, $\mathcal{S}_\varphi=\pi_0(Z_{\widehat{G}}(\varphi))$, and $\psi\colon F\to \C^\times$ is an additive character of order zero.
\end{theorem}

The contents of this paper are the following: In Section~\ref{sect:Preliminaries}, we explain two important topics, which will be used later. 
One is \emph{Lusztig's Jordan decomposition}, which relates irreducible representations of finite groups of Lie type with unipotent ones of smaller groups. 
The other is the local Langlands correspondence for unipotent supercuspidal representations mentioned above. We also introduce some basic concepts on the local Langlands correspondence in more general settings. 
In Section~\ref{sec:const_of_H}, we construct a reductive group $H$ and an unramified L-parameter for it from an L-parameter for $G$ which admits a cuspidal enhancement. Then cuspidal enhancements of the L-parameters of them correspond to each other.
Also, we construct an embedding of apartments of (inner twists of) $H$ into those of $G$ so that we can relate their parahoric subgroups with each other. 
In Section~\ref{sect:const_of_corr}, we use Lusztig's Jordan decomposition and relate unipotent supercuspidal representations of $H$ to a certain family of depth-zero supercuspidal representations of $G$. Then the local Langlands correspondence for unipotent supercuspidal representations of $H$ provides a correspondence $\LLC$ for depth-zero supercuspidal representations of $G$. 
The content of Section~\ref{sect:action_fund_parahoric} is independent of the previous sections. In this section, we examine the action of the fundamental group of $G$ on the reductive quotient of their parahoric subgroups. This result is used in the next section. 
In Section~\ref{sect:corr_for_simple}, we check for some simple adjoint groups that the map $\LLC$ constructed in Section~\ref{sect:const_of_corr} is bijective. 
Finally, in Section~\ref{sect:HII_conj} we prove the Hiraga--Ichino--Ikeda conjecture for our $\LLC$ when it is bijective. 
In fact, this follows from the fact that this conjecture holds for the correspondence of unipotent supercuspidal representations: As we construct $\LLC$ by relating depth-zero supercuspidal representations (resp.\ depth-zero cuspidal enhanced L-parameters) with unipotent supercuspidal representations (resp.\ unramified cuspidal enhanced L-parameters) of smaller groups, we can prove the equality of Theorem~\ref{thm:main2} by comparing the formal degree (resp.\ adjoint $\gamma$-factor) of them.

\section*{Acknowledgments}

I would like to thank my advisor Yoichi Mieda for his constant support and encouragement. This work was supported by the WINGS-FMSP program at the Graduate School of Mathematical Sciences, the University of Tokyo.

\section{Preliminaries}\label{sect:Preliminaries}

In this section, we will explain two topics: Lusztig's Jordan decomposition and the local Langlands correspondence for unipotent cuspidal representations, both of which play a crucial role in this paper.

\subsection{Lusztig's Jordan decomposition}

Let $\mathbb{G}$ be a connected reductive group over a finite field $k$. For a maximal torus $\mathbb{T}\subset \mathbb{G}$ and a character $\theta\colon \mathbb{T}(k)\to \C^\times$, Deligne and Lusztig defined a virtual character $R_\mathbb{T}^\mathbb{G}(\theta)$  (called the \emph{Deligne--Lusztig character}) of $\mathbb{G}(k)$ in \cite[1.20, Corollary 4.3]{DL}.

\begin{proposition}[{\cite[Corollary 7.7]{DL}}]
    For any irreducible representation $\rho$ of $\mathbb{G}(k)$, there exists a pair $(\mathbb{T},\theta)$ such that $\rho$ is an irreducible component of $R_\mathbb{T}^\mathbb{G}(\theta)$.
\end{proposition}

Let $\widehat{\mathbb{G}}$ be the dual group of $\mathbb{G}$ over $k$. Then for each $\mathbb{T}\subset \mathbb{G}$, the dual $\widehat{\mathbb{T}}$ can be considered as a maximal torus of $\widehat{\mathbb{G}}$ determined up to $\widehat{\mathbb{G}}(k)$-conjugacy. 
For a finite group $\Gamma$, we denote by $\Irr(\Gamma)$ the set of isomorphism classes of irreducible representations of $\Gamma$.
If we take an embedding $i\colon \overline{k}^\times\hookrightarrow \C^\times$, we can construct an isomorphism
\[
    \Irr(\mathbb{T}(k))\cong \widehat{\mathbb{T}}(k)
\]
as in \cite[(5.2.4)]{DL} (we use the canonical isomorphism $\Q/\Z\cong \mu_\infty(\C);\ r\mapsto e^{2\pi\sqrt{-1}r}$). Thus $\theta$ provides an element $s\in \widehat{\mathbb{T}}(k)$ and we have a bijection
\[
    \{(\mathbb{T},\theta)\mid \mathbb{T}\subset \mathbb{G},\ \theta\in \Irr(\mathbb{T}(k))\}/\mathord{\sim}\cong \{(\widehat{\mathbb{T}},s)\mid \widehat{\mathbb{T}}\subset\widehat{\mathbb{G}},\ s\in \widehat{\mathbb{T}}(k)\}/\mathord{\sim},
\]
where $\sim$ in the left and right hand side are the conjugacy equivalence under $\mathbb{G}(k)$ and $\widehat{\mathbb{G}}(k)$, respectively. Since the virtual character $R_\mathbb{T}^\mathbb{G}(\theta)$ depends only on the $\mathbb{G}(k)$-conjugacy class of $(\mathbb{T},\theta)$, we also denote it by $R_{\widehat{\mathbb{T}}}^\mathbb{G}(s)$. Then we can state the \emph{exclusion theorem}:

\begin{theorem}[{\cite[(7.5.2)]{Lusztig_exclusion}}]
    If $R_{\widehat{\mathbb{T}}}^\mathbb{G}(s)$ and $R_{\widehat{\mathbb{T}'}}^\mathbb{G}(s')$ have an irreducible component in common, then $s$ and $s'$ are conjugate under $\widehat{\mathbb{G}}(k)$.
\end{theorem}

For $s\in \widehat{\mathbb{G}}(k)$, we define the \emph{Lusztig series} with respect to $s$ as:
\[
    \Irr(\mathbb{G}(k))_s\coloneq \{\rho\in \Irr(\mathbb{G}(k))\mid \text{$\rho$ appears in $R_{\widehat{\mathbb{T}}}^\mathbb{G}(s)$ for some $\widehat{\mathbb{T}}\ni s$}\}.
\]
Then the above theorem gives a partition
\[
    \Irr(\mathbb{G}(k))=\coprod_{[s]}\Irr(\mathbb{G}(k))_s,
\]
where $[s]$ runs over $\widehat{\mathbb{G}}(k)$-conjugacy classes of semisimple elements. If $s=1\in \widehat{\mathbb{G}}(k)$, we also denote the corresponding Lusztig series by $\Irr(\mathbb{G}(k))_\unip$ and representations inside it are called \emph{unipotent representations}. Even for a disconnected reductive group $\mathbb{G}'$ over $k$, we can define $\Irr(\mathbb{G}'(k))_\unip$ by setting
\[
    \rho\in \Irr(\mathbb{G}'(k))_\unip\iff \rho|_{\mathbb{G}'^\circ(k)}=\bigoplus_i \rho_i\quad\text{for some }\rho_i\in \Irr(\mathbb{G}'^\circ(k))_\unip,
\]
where $\mathbb{G}'^\circ$ is the identity component of $\mathbb{G}'$. Also, we can define 
\[
    R_{\widehat{\mathbb{T}}}^{\mathbb{G}'}(1)=\operatorname{Ind}_{\mathbb{G}'^\circ(k)}^{\mathbb{G}'(k)}R_{\widehat{\mathbb{T}}}^{\mathbb{G}'^\circ}(1),
\]
whose irreducible components are unipotent representations. Now the following is \emph{Lusztig's Jordan decomposition}:

\begin{theorem}[{\cite[12]{Lusztig_discon}}]\label{thm:Lusztig_jord_decomp}
    There exists a bijection
    \[
        J_s\colon \Irr(\mathbb{G}(k))_s\to \Irr(Z_{\widehat{\mathbb{G}}}(s)(k))_\unip,
    \]
    such that for any $\widehat{\mathbb{T}}\ni s$, the multiplicity of $\rho\in \Irr(\mathbb{G}(k))_s$ in $(-1)^{\sigma(\mathbb{G})}R_{\widehat{\mathbb{T}}}^\mathbb{G}(s)$ is equal to that of $J_s(\rho)$ in $(-1)^{\sigma(Z_{\widehat{\mathbb{G}}}(s)^\circ)}R_{\mathbb{T}}^{Z_{\widehat{\mathbb{G}}}(s)}(1)$, where $\sigma(\mathbb{G})$ and $\sigma(Z_{\widehat{\mathbb{G}}}(s)^\circ)$ are the split rank of $\mathbb{G}$ and $Z_{\widehat{\mathbb{G}}}(s)^\circ$, respectively.
\end{theorem}

\begin{remark}
    In particular, we have a bijection
    \[
        \Irr(\mathbb{G}(k))_\unip\to \Irr(\widehat{\mathbb{G}}(k))_\unip.
    \]
    In fact, the set $\Irr(\mathbb{G}(k))_\unip$ depends only on the type of $\mathbb{G}$. That is, for any isogeny $\mathbb{G}\to \mathbb{G}'$, restriction to $\mathbb{G}(k)$ gives a bijection $\Irr(\mathbb{G}'(k))_\unip\to \Irr(\mathbb{G}(k))_\unip$ (\cite[Proposition 7.10]{DL}).
\end{remark}

This bijection is, in general, not canonical. If the center of $\mathbb{G}$ is connected, Digne and Michel determine a unique bijection by imposing several conditions (\cite[Theorem 7.1]{DM}), which makes $J_s$ equivariant under outer automorphisms of $\mathbb{G}$ (\cite[Theorem 3.1]{CS_equiv_conn}). In particular, the correspondence $\Irr(\mathbb{G}(k))_\unip\cong \Irr(\widehat{\mathbb{G}}(k))_\unip$ is equivariant for general $\mathbb{G}$; indeed, we can reduce it to the case where $\mathbb{G}$ is adjoint.

Also, Lusztig's Jordan decomposition is compatible with the cuspidality of representations as follows:

\begin{theorem}[{\cite[Theorem 3.2.22]{Geck_Malle_2020}}]
    Let $s\in \widehat{\mathbb{G}}(k)$ be a semisimple element and $\mathbb{H}=Z_{\widehat{\mathbb{G}}}(s)$. Suppose that $Z(\mathbb{H}^\circ)$ has the same split rank as that of $Z(\mathbb{G})$. Then, for any $\rho\in \Irr(\mathbb{G}(k))_s$, $\rho$ is cuspidal if and only if $J_s(\rho_0)$ is cuspidal. Moreover, if the split rank of $Z(\mathbb{H}^\circ)$ and $Z(\mathbb{G})$ is different, then any $\rho\in \Irr(\mathbb{G}(k))_s$ is not cuspidal.
\end{theorem}

Since any representation of a torus is cuspidal, we have the following corollary:

\begin{corollary}\label{cor:regular_cuspidal}
    Suppose that $Z_{\widehat{\mathbb{G}}}(s)^\circ$ is an elliptic maximal torus. Then $\Irr(\mathbb{G}(k))_s$ only consists of cuspidal representations.
\end{corollary}

\subsection{The local Langlands correspondence for unipotent supercuspidal representations}

Let $G$ be a quasi-split reductive group over a non-Archimedean local field $F$. Then we can consider the dual group $\widehat{G}$ defined over $\C$ on which the absolute Galois group $\Gamma_F$ of $F$ acts. We define $\Ldual{G}\coloneq \widehat{G}\rtimes W_F$ and call it the \emph{Langlands dual} of $G$, where $W_F\subset \Gamma_F$ is the Weil group of $F$. In this paper, we only consider the case where $G$ splits over an unramified extension of $F$. In this case, the inertia group $I_F\subset W_F$ acts on $\widehat{G}$ trivially, so we instead define the Langlands dual of $G$ as $\widehat{G}\rtimes \langle \Fr\rangle$, where $\Fr\in W_F/I_F$ is the arithmetic Frobenius element.

The local Langlands correspondence is, roughly speaking, a parametrization of irreducible representations of $G(F)$ in terms of L-parameters $\varphi\colon W_F\times\SL_2(\C)\to \Ldual{G}$ of $\Ldual{G}$. However, this conjectural parametrization is not bijective in general, and we need to \emph{enhance} L-parameters to obtain a bijective correspondence. We explain this idea in the framework of pure inner forms.

Let $\varphi\colon W_F\times\SL_2(\C)\to \Ldual{G}$ be an L-parameter of $\Ldual{G}$. We define
\[
    \mathcal{S}_\varphi\coloneq \pi_0(Z_{\widehat{G}}(\varphi)).
\]
An \emph{enhanced L-parameter} is a pair $(\varphi,\rho)$, where $\varphi$ is an L-parameter and $\rho$ is an irreducible representation of $\mathcal{S}_\varphi$. Since $Z_{\widehat{G}}(\varphi)$ contains $Z(\widehat{G})^{\Gamma_F}$ as a central group, $\rho|_{Z(\widehat{G})^{\Gamma_F}}$ defines a character $\omega_\rho$ of $\pi_0(Z(\widehat{G})^{\Gamma_F})$. There exists a canonical isomorphism
\[
    \Irr(\pi_0(Z(\widehat{G})^{\Gamma_F}))\cong H^1(\Gamma_F,G),
\]
and each $\xi\in H^1(\Gamma_F,G)$ defines an inner form $G_\xi$ of $G$ via $H^1(\Gamma_F,G)\to H^1(\Gamma_F,G_\mathrm{ad})$. For an inner form $G'$ of $G$ and $\xi\in H^1(\Gamma_F,G)$ with $G_\xi=G'$, we define $\Phi_\mathrm{e}(G')_\xi$ as the set of enhanced L-parameters $(\varphi,\rho)$ such that the central character $\omega_\rho$ of $\rho$ corresponds to $\xi$. If $G$ is adjoint, then $G'$ specifies a unique $\xi$ and we will abbreviate it as $\Phi_\mathrm{e}(G')$.

\begin{conjecture}\label{conj:LLC}
    There exists a bijection
    \[
        {\Irr(G'(F))} \to {\Phi_\mathrm{e}(G')_\xi/\mathord{\sim}},
    \]
    where $\sim$ in the right hand side is the conjugacy class equivalence under $\widehat{G}$. Also, it is canonically defined from a Whittaker datum for $G$ and satisfies several properties.
\end{conjecture}

In the side of $\Irr(G'(F))$, we know an important class called \emph{supercuspidal representations}. In \cite[Definition 6.9]{AMS}, Aubert, Moussaoui and Solleveld defined the concept of \emph{cuspidal enhanced L-parameters} and conjectured that the bijection of Conjecture~\ref{conj:LLC} gives a bijection between the isomorphism classes of supercuspidal irreducible representations and the conjugacy classes of cuspidal enhanced L-parameters.

Another concept that can be considered in both sides of the correspondence is \emph{depth}. It is expected that via the bijection in Conjecture~\ref{conj:LLC}, at least depth-zero irreducible representations and  depth-zero enhanced L-parameters correspond to each other. Here, $\pi\in \Irr(G'(F))$ is \emph{of depth zero} if $\pi$ has a nonzero $G'(F)_{x,0+}$-fixed vector for some point $x$ of the (reduced) Bruhat--Tits building $\mathcal{B}(G')$ of $G'$. Also, an enhanced L-parameter $(\varphi,\rho)$ is \emph{of depth zero} if $\varphi|_{P_F}$ is trivial, where $P_F\subset I_F$ is the wild inertia subgroup.

If $\pi\in \Irr(G'(F))$ is of depth zero, then $\pi|_{G'(F)_{x,0}}$ contains irreducible representations of $\mathbb{G}(k)\coloneq G'(F)_{x,0}/G'(F)_{x,0+}$. It is known that $\mathbb{G}(k)$ can be naturally seen as the $k$-rational points of a reductive group $\mathbb{G}$ over $k$, where $k$ is the residue field of $F$. We say $\pi$ is \emph{unipotent} if $\pi$ contains an unipotent irreducible representation of $\mathbb{G}(k)$ for some $x\in\mathcal{B}(G')$. On the other hand, we say an enhanced L-parameter $(\varphi,\rho)$ is \emph{unramified} if $\varphi|_{I_F}$ is trivial. Feng, Opdam and Solleveld provide a correspondence for unipotent supercuspidal representations:

\begin{theorem}[{\cite[Theorem 2]{FOS}}]\label{thm:LLC_for_unipotent}
    Let $\Irr(G'(F))_{\unip,\cusp}\subset \Irr(G'(F))$ be the set of unipotent supercuspidal representations and $\Phi_\mathrm{e}(G')_{\xi,\mathrm{nr},\cusp}\subset \Phi_\mathrm{e}(G')_{\xi}$ be the set of unramified cuspidal enhanced L-parameters. Then there exists a bijection
    \[
        {\Irr(G'(F))_{\unip,\cusp}}\to {\Phi_\mathrm{e}(G')_{\xi,\mathrm{nr},\cusp}}/\mathord{\sim},
    \]
    which satisfies several properties.
\end{theorem}

\section{The reductive group \textit{H} and its apartments}\label{sec:const_of_H}

Let $F$ be a non-Archimedean local field, $k=\mathbb{F}_q$ its residue field with characteristic $p>0$, and $G$ an unramified adjoint group over $F$. In this section, we will construct an unramified reductive group $H$ from a discrete L-parameter for $G$ of depth zero. We will also relate the apartments of its inner twists with those of inner twists of $G$.

\begin{definition}
    Let $\Phi(\Ldual{G})_{0,\disc}$ be the set of discrete L-parameters of $\Ldual{G}$ of depth zero. We define an equivalent relation $\stackrel{\mathrm{w}}{\sim}$ on $\Phi(G)_{0,\disc}$ as
    \[
        \varphi\stackrel{\mathrm{w}}{\sim} \varphi'\iff \varphi|_{I_F}=\varphi'|_{I_F}\ \text{and}\  \varphi'(\Fr)=z\cdot \varphi(\Fr)\ \text{for some }z\in Z(\widehat{H}),
    \]
    where $\widehat{H}\coloneq Z_{\widehat{G}}(\varphi|_{I_F})=Z_{\widehat{G}}(\varphi'|_{I_F})$ and $\Fr\in W_F$ is a lift of the Frobenius. The choice of $\Fr$ does not affect the definition of $\stackrel{\mathrm{w}}{\sim}$, and we denote the equivalence class of $\varphi$ by $[\varphi]$.
\end{definition}
Note that $\widehat{H}$ in the above definition is a connected reductive group. Indeed, since $\varphi$ is of depth zero, $\varphi(I_F)$ is generated by a semisimple element $\varphi(\zeta)$, where $\zeta$ is a topological generator of $I_F/P_F$. Also, $\widehat{G}$ is simply connected by assumption. Thus the semisimple centralizer $\widehat{H}=Z_{\widehat{G}}(\varphi(\zeta))$ is connected and reductive (\cite[8.1. Theorem]{St}). Since $\widehat{H}$ is determined by the equivalence class of $\varphi$, we sometimes write it by $\widehat{H}_{[\varphi]}$.

By definition, $\widehat{H}$ is normalized by $\varphi(\Fr)$. Take a Borel pair of $\widehat{H}$ and a pinning of it. Then there exists $h\in \widehat{H}$ such that $h\cdot \varphi(\Fr)$ stabilizes this pinning. We define an L-group $\Ldual{H}=\widehat{H}\rtimes\langle \Fr\rangle$ as $\Fr$ acts on $\widehat{H}$ by $h\cdot \varphi(\Fr)$. Then we can extend the inclusion $\widehat{H}\subset \widehat{G}$ to 
\[
    \Ldual{j}\colon \Ldual{H}\to \Ldual{G};\quad \Fr\mapsto h\cdot \varphi(\Fr),
\]
and get a factorization
\[
    \varphi=\Ldual{j}\circ\varphi_H\colon W_F\times \SL_2(\C)\to \Ldual{H}\to \Ldual{G}.
\]
Since $h$ above is unique up to multiplication by $Z(\widehat{H})$, $\Ldual{H}$ and $[\varphi_H]$ does not depend on the choice of $h$. Moreover, we have
\[
    Z_{\widehat{G}}(\varphi)=Z_{\widehat{H}}(\varphi_H).
\]
Since $\varphi$ is discrete and $G$ is semisimple, $Z_{\widehat{G}}(\varphi)$ is finite. Thus $Z(\widehat{H})^\Fr\subset Z_{\widehat{H}}(\varphi_H)$ is also finite.

Let $H=H_{[\varphi]}$ be a quasi-split reductive group whose Langlands dual is $\Ldual{H}$. Then the inclusion $Z(\widehat{G})^\Fr\subset Z(\widehat{H})^\Fr$ gives a surjection
\begin{equation}
    H^1(\Gamma_F,H)\twoheadrightarrow H^1(\Gamma_F,G).\label{eq:corr_of_twists}
\end{equation}
Let $S\subset H$ be a \emph{quasi-split} maximal torus, i.e.,\ a Levi subgroup of a Borel subgroup of $H$. Then its Langlands dual $\Ldual{S}$ is canonically embedded in $\Ldual{H}$ up to $\widehat{H}$-conjugacy, so we have an embedding
\[
    \widehat{S}\hookrightarrow\widehat{H}\hookrightarrow \widehat{G}
\]
whose conjugacy class is stable under the action of $\Fr$. This induces a $G$-conjugacy class of embeddings $S\hookrightarrow G$ which is $\Gamma_F$-stable.

\begin{lemma}\label{lem:embed_torus}
    Let $\mathcal{B}(G)$ be the (reduced) Bruhat--Tits building of $G$ and $x\in \mathcal{B}(G)$ a hyperspecial vertex. Then we can choose an embedding $i_x\colon S\hookrightarrow G$ from the above conjugacy class such that $i_x$ is defined over $F$ and that the apartment $\mathcal{A}(S,G)$ contains $x$. Moreover, such an embedding is unique up to $G(F)_{x,0}$-conjugacy.
\end{lemma}

\begin{proof}
    In \cite[Lemma 3.4.12]{Kaletha_reg}, the claim is proved when $S$ is elliptic, but the same proof can be applied as long as $S$ splits over $F^\ur$. 
\end{proof}

Since $G$ is adjoint, any hyperspecial vertex is conjugate to each other under $G(F)$. Thus we can determine the embedding $i_x$ up to $G(F)$-conjugacy. 
In the following, we fix such a vertex $x_0$ and a corresponding embedding $i_0$.

\begin{proposition}\label{prop:embd_apart}
    Let $x_{0,H}\in \mathcal{A}(S,H)$ be a hyperspecial vertex. Then the isomorphism of affine spaces over $V(S)=X_\ast(S)\otimes \R$:
    \[
        \iota\colon \widetilde{\mathcal{A}}(S_{F^\ur},H_{F^\ur})=\mathcal{A}(S_{F^\ur},H_{F^\ur})\times V(Z(H)^\circ)\to \mathcal{A}(S_{F^\ur},G_{F^\ur}),
    \]
    defined by setting $\iota(x_{0,H},0)=x_0$, is equivariant under the action of $\Fr$ and induces an inclusion of the $(\text{absolute})$ affine root systems
    \[
        \Psi_{S,H}=\{\psi\in \Psi_{S,G}\mid \dot{\psi}\in \Phi_{S,H}\}\subset \Psi_{S,G},
    \]
    where $\Phi_{S,H}$ is the finite root system of $(S,H)$.
\end{proposition}

\begin{proof}
    The $\Fr$-equivariance is clear. Moreover, since $x_{0,H}$ (resp.\ $x_0$) is special as a vertex of $\mathcal{A}(S_{F^\ur},H_{F^\ur})$ (resp.\ $\mathcal{A}(S_{F^\ur},G_{F^\ur})$), both of 
    \[
        \{r\in \R\mid U_{H,\alpha}(F^\ur)_{x_0,r+}\subsetneq U_{H,\alpha}(F^\ur)_{x_0,r}\}
    \]
    and
    \[
        \{r\in \R\mid U_{G,\alpha}(F^\ur)_{x_0,r+}\subsetneq U_{G,\alpha}(F^\ur)_{x_0,r}\}
    \]
    coincide with $\Z$ for any $\alpha\in \Phi_{S,H}$, where $U_{H,\alpha}\subset H,\ U_{G,\alpha}\subset G$ are the root subgroups. This implies that $ \Psi_{S,H}=\{\psi\in \Psi_{S,G}\mid \dot{\psi}\in \Phi_{S,H}\}$.
\end{proof}

We also fix a hyperspecial vertex $x_{0,H}\in \mathcal{A}(S,H)$ and denote the corresponding embedding by $\iota_0$.

\begin{corollary}\label{cor:corr_weyl}
    The map $\iota_0$ gives a $\Fr$-equivariant inclusion 
    \[
        N_H(S)(F^\ur)/S(F^\ur)_0\hookrightarrow N_G(S)(F^\ur)/S(F^\ur)_0.
    \]
\end{corollary}
\begin{proof}
    Since $N_H(S)(F^\ur)$ acts on $\widetilde{\mathcal{A}}(S_{F^\ur},H_{F^\ur})$ and its kernel is $S(F^\ur)_0$, we can consider $N_H(S)(F^\ur)/S(F^\ur)_0$ as a subgroup of $\Aut(\widetilde{\mathcal{A}}(S_{F^\ur},H_{F^\ur}))$. Then we can express
    \[
        N_H(S)(F^\ur)/S(F^\ur)_0=S(F^\ur)/S(F^\ur)_0\cdot W(\Psi_{S,H}).
    \]
    Similarly we have $N_G(S)(F^\ur)/S(F^\ur)_0=S(F^\ur)/S(F^\ur)_0\cdot W(\Psi_{S,G})$. Since $\iota_0$ induces an inclusion $W(\Psi_{S,H})\hookrightarrow W(\Psi_{S,G})$, it also provides an inclusion as in the statement.
\end{proof}

Let $\Omega_H,\Omega_G$ be the fundamental groups of $H,G$ respectively, and $\kappa_H,\kappa_G$ the Kottwitz morphisms. Since $\kappa_H,\kappa_G$ are trivial on $S(F^\ur)_0$, we get a diagram
\begin{equation}
    \begin{tikzcd}
        N_H(S)(F^\ur)/S(F^\ur)_0\arrow[r,hook]\arrow[d,"\kappa_H",two heads]&N_G(S)(F^\ur)/S(F^\ur)_0\arrow[d,"\kappa_G",two heads]\\
        \Omega_H\arrow[r,two heads]&\Omega_G,
        \end{tikzcd}\label{eq:comm_weyl_fund}
\end{equation}
where the bottom horizontal arrow is the canonical surjection
\[
    \Omega_H=X_\ast(S)/\langle \Phi_{S,H}^\vee\rangle\twoheadrightarrow X_\ast(S)/\langle \Phi_{S,G}^\vee\rangle=\Omega_G.
\]
In fact, this diagram commutes; it is clear when restricting the top arrow to $S(F^\ur)/S(F^\ur)_0\xrightarrow{=}S(F^\ur)/S(F^\ur)_0$. Moreover, $W(\Psi_{S,H})$ and $W(\Psi_{S,G})$ are killed by $\kappa_H,\kappa_G$ respectively. Since $N_H(S)(F^\ur)/S(F^\ur)_0$ is generated by $S(F^\ur)/S(F^\ur)_0$ and $W(\Psi_{S,H})$, the image of $\kappa_H(n)$ in $\Omega_G$ is $\kappa_G(n)$ for any $n\in N_H(S)(F^\ur)/S(F^\ur)_0$, so the diagram \eqref{eq:comm_weyl_fund} commutes.
Also, the map $\Omega_H\twoheadrightarrow \Omega_G$ induces
\[
    H^1(\Gamma_F,H)\cong H^1(\Gamma_{F^\ur/F},\Omega_H)\to H^1(\Gamma_{F^\ur/F},\Omega_G)\cong H^1(\Gamma_F,G),
\]
which is the same map as \eqref{eq:corr_of_twists} (we denote $\mathrm{Gal}(F^\ur/F)$ by $\Gamma_{F^\ur/F}$).

Now we will generalize Proposition~\ref{prop:embd_apart} to the case when $H,G$ are replaced by their inner twists. For simplicity, we may assume that any inner twist $f\colon G_{F^\ur}\to G'_{F^\ur}$ considered in the paper satisfies that $G_{F^\ur}=G'_{F^\ur}$ and $f=\Ad(g)$ for some $g\in G(F^\ur)$.

\begin{proposition}
    Let $\xi\in H^1(\Gamma_F,H)$ and $\overline{\xi}\in H^1(\Gamma_F,G)$ be the image of $\xi$ along the map \eqref{eq:corr_of_twists}. Also, let $H_\xi,G_{\overline{\xi}}$ be the corresponding inner forms and $S'\subset H_\xi$ an unramified maximally split maximal torus. Then we have an embedding $i_\xi\colon S'\hookrightarrow G_{\overline{\xi}}$ defined over $F$ and a $\Fr$-stable map
    \[
        \iota_\xi\colon \widetilde{\mathcal{A}}(S'_{F^\ur},H_{\xi,F^\ur})\to \mathcal{A}(S'_{F^\ur},G_{\overline{\xi},F^\ur})
    \]
    which is isomorphic as a morphism of affine spaces over $V(S')$ and induces an inclusion
    \[
        \Psi_{S',H_{\xi}}=\{\psi\in \Psi_{S',G_{\overline{\xi}}}\mid \dot{\psi}\in \Phi_{S',H_\xi}\}\subset \Psi_{S',G_{\overline{\xi}}}.
    \]
\end{proposition}

\begin{proof}
    We have a canonical isomorphism $H^1(\Gamma_F,H)\cong H^1(\Gamma_{F^\ur/F},H(F^\ur))$. We express a continuous cocycle $c\colon \Gamma_{F^\ur/F}\to H(F^\ur)$ by its value $c(\Fr)$ at $\Fr$. For a given $\xi\in H^1(\Gamma_F,H)$, we can choose an inner twist $f\colon H_{F^\ur}\xrightarrow{\sim} H_{\xi,F^\ur}$ and $n_\xi \in H(F^\ur)$ such that $f(S)=S',\ f^{-1}\circ \Fr(f)=\mathrm{Ad}(n_\xi)$ and the cohomology class of $n_\xi$ is $\xi$. In this case, $n_\xi$ belongs to $N_H(S)(F^\ur)$. We denote by $\overline{n}_\xi$ the image of $n_\xi$ in $N_H(S)(F^\ur)/S(F^\ur)_0$. Via the inclusion in Corollary~\ref{cor:corr_weyl}, we can consider $\overline{n}_\xi$ as an element of $N_G(S)(F^\ur)/S(F^\ur)_0$. Choose a lift $n'_\xi\in N_G(S)(F^\ur)$ of $\overline{n}_\xi$. Then commutativity of the diagram \eqref{eq:comm_weyl_fund} implies that we have an inner twist $f_G\colon G_{F^\ur}\to G_{\overline{\xi},F^\ur}$ such that $f^{-1}\Fr(f)=\Ad(n'_\xi)$. Since the action of $n_\xi$ and $n'_\xi$ on $S$ are the same, the composite
    \[
        i_\xi=f_G\circ i_0\circ (f|_S)^{-1}\colon S'_{F^\ur}\to G_{\overline{\xi},F^\ur}
    \]
    is in fact defined over $F$. Also, $f$ and $f_G$ induce isomorphisms
    \[
        f_\ast\colon \widetilde{\mathcal{A}}(S_{F^\ur},H_{F^\ur})\to \widetilde{\mathcal{A}}(S'_{F^\ur},H_{\xi,F^\ur})
    \]
    and 
    \[
        f_{G\ast}\colon \mathcal{A}(S_{F^\ur},G_{F^\ur})\to \mathcal{A}(S'_{F^\ur},G_{\overline{\xi},F^\ur})
    \]
    respectively. Thus we can construct a morphism
    \[
        \iota_\xi=f_{G\ast}\circ \iota_0\circ (f_\ast)^{-1}\colon \widetilde{\mathcal{A}}(S'_{F^\ur},H_{\xi,F^\ur})\to \mathcal{A}(S'_{F^\ur},G'_{\overline{\xi},F^\ur}),
    \]
    which induces an inclusion $\Psi_{S',H_\xi}\hookrightarrow \Psi_{S',G_{\overline{\xi}}}$. Its $\Fr$-equivariance follows from the fact that $\iota_0$ is equivariant under the action of $N_H(S)(F^\ur)/S(F^\ur)_0$ and $\Fr$, and that the action of $n_\xi$ and $n'_\xi$ on the apartment are the same.
\end{proof}

\begin{remark}
    The construction of $(i_\xi,\iota_\xi)$ depends on the choice of $(f,n_\xi),n'_\xi$ and $f_G$. However, its $G_{\overline{\xi}}(F)$-conjugacy class is independent; changing $f_G$ clearly does not affect the conjugacy class, and neither does $n'_\xi$ because $H^1(\Gamma_{F^\ur/F},S(F^\ur)_0)$ is trivial. Also, if we change the pure inner twist $(f,n_\xi)$ with $(f\circ \Ad(n),n^{-1}n_\xi \Fr(n))$ where $n\in N_H(S)(F^\ur)$, then $i_\xi$ and $\iota_\xi$ change into $\Ad(f_G(n'))\circ i_\xi\circ\Ad(f(n))$ and $\Ad(f_G(n'))\circ\iota_\xi\circ \Ad(f(n))$ respectively, where $n'\in N_G(S)(F^\ur)$ has the same image in $N_G(S)(F^\ur)/S(F^\ur)_0$ as that of $n$. This pair is indeed the same as $(i_\xi,\iota_\xi)$ because of the equivariance of $\iota_\xi$.
\end{remark}

Restricting $\iota_\xi$ to the $\Fr$-fixed parts, we obtain 
\[
    \iota_\xi\colon \mathcal{A}(S',H_\xi)\to \mathcal{A}(S',G_{\overline{\xi}})
\]
(recall that $Z(\widehat{H})^{\Fr,\circ}$ is trivial, which means that $Z(H_\xi)^\circ$ is anisotropic and $V(Z(H_\xi)^\circ)^\Fr=\{0\}$).

\begin{lemma}
    The map $\iota_\xi$ sends any vertex of $\mathcal{A}(S',H_\xi)$ to a vertex of $\mathcal{A}(S',G_{\overline{\xi}})$.
\end{lemma}

\begin{proof}
    For $x\in \mathcal{A}(S',H_\xi)$, put
    \[
        \Psi_{S',H_\xi,x}=\{\psi\in \Psi_{S',H_\xi}\mid \psi(x)=0\}.
    \]
    Then the absolute root system of $\mathbb{H}_{\xi,x}=H_{\xi,x,0}/H_{\xi,x,0+}$ with respect to the maximal torus $\mathbb{S}'=S'_0/S'_{0+}$ coincides with (the derivation of) $\Psi_{S',H_\xi,x}$ under the canonical identification $X_\ast(S')\cong X_\ast(\mathbb{S}')$ (\cite[Theorem 8.4.10]{Kaletha_Prasad_2023}). Let $\mathbb{S}'_0\subset \mathbb{S}'$ be the maximal split torus. Since $Z(H_\xi)^\circ$ is anisotropic, $x$ is a vertex of $\mathcal{A}(S',H_\xi)$ if and only if $\Psi_{S',H_\xi,x}|_{\mathbb{S}'_0}$ spans $X^\ast(\mathbb{S}'_0)\otimes\R$. 
    We define $\Psi_{S',G_{\overline{\xi}},x}\subset \Psi_{S',G_{\overline{\xi}}}$ similarly, then $\Psi_{S',H_\xi,x}\subset \Psi_{S',G_{\overline{\xi}},x}$. Thus if $x$ is a vertex in $\mathcal{A}(S',H_\xi)$, so is it in $\mathcal{A}(S',G_{\overline{\xi}}).$
\end{proof}

\section{Construction of the correspondence}\label{sect:const_of_corr}

In this section, we will relate unipotent supercuspidal representations of $H$ to depth-zero supercuspidal representations of $G$, which enables us to parameterize them by L-parameters via the local Langlands correspondence for unipotent supercuspidal representations under some assumption.

In the previous section, we start from $\varphi\in \Phi(\Ldual{G})_{0,\disc}$ and construct an unramified reductive group $H=H_{[\varphi]}$, and also embeddings $i_\xi\colon S'\hookrightarrow G_{\overline{\xi}},\ \iota_\xi\colon \mathcal{A}(S',H_\xi)\hookrightarrow \mathcal{A}(S',G_{\overline{\xi}})$ for each $\xi\in H^1(\Gamma_F,H)$ and an unramified maximally split maximal torus $S'\subset H_\xi$. Next, we will construct a character $\theta$ of $\mathbb{S}'(k)=S'(F)_0/S'(F)_{0+}$.

Recall that $\varphi(I_F)\subset Z(\widehat{H})\subset \widehat{S}\cong X_\ast(\widehat{S})\otimes \C^\times$. Moreover, $\varphi|_{I_F}$ factors through $I_F/P_F$, which is a profinite group with no $p$-torsion, so the image of $I_F$ is inside  $X_\ast(\widehat{S})\otimes (\C^\times)_{p'\text{-tors}}\cong X^\ast(S)\otimes (\C^\times)_{p'\text{-tors}}$. If we take an embedding $i\colon \overline{k}^\times\cong (\C^\times)_{p'\text{-tors}}\subset \C^\times$, we have isomorphisms
\[
    X^\ast(S)\otimes (\mathbb{C}^\times)_{p'\text{-tors}}\cong X^\ast(\mathbb{S})\otimes \overline{k}^\times\cong X_\ast(\widehat{\mathbb{S}})\otimes \overline{k}^\times \cong \widehat{\mathbb{S}}.
\]
On the other hand, $i$ provides a system of generators $\{i^{-1}(e^{2\pi\sqrt{-1}/n})\in \mu_n(\overline{k})\}_{(n,p)=1}$. This determines a topological generator $\zeta_i$ of
\[
    I_F/P_F\cong \varprojlim_{(n,p)=1} \mu_n(\overline{k}),
\]
and $s_i=\varphi(\zeta_i)$ can be considered as an element of $\widehat{\mathbb{S}}$.

\begin{lemma}
    In the above setting, $s_i\in \widehat{\mathbb{S}}(k)$.
\end{lemma}

\begin{proof}
    First, we see how the (arithmetic) Frobenius acts on the character and cocharacter groups: over the local field $F$, $\Fr$ induces an automorphism $f\colon X_\ast(S)\to X_\ast(S)$, and acts on $X^\ast(S)$ by ${}^t f^{-1}.$ Also, $\Fr$ acts on $X_\ast(\widehat{S})$ so that the canonical isomorphism $X_\ast(\widehat{S})\cong X^\ast(S)$ is $\Fr$-equivariant. On the other hand, when considering $f$ as an automorphism of $X_\ast(\mathbb{S})$, the $k$-structure of $\widehat{\mathbb{S}}$ is defined so that $\Fr$ acts on $X_\ast(\widehat{\mathbb{S}})$ as ${}^t f$ under the identification $X_\ast(\widehat{\mathbb{S}})\cong X^\ast(\mathbb{S})$ (\cite[5.2]{DL}). Thus, the isomorphism $X_\ast(\widehat{S})\cong X_\ast(\widehat{\mathbb{S}})$ is \emph{anti-equivariant} under $\Fr$.

    As an element of $\widehat{S}$, we have
    \[
        \Fr^{-1}(s_i)=\Ad(\varphi_H(\Fr^{-1}))(\varphi_H(\zeta_i))=\varphi_H(\zeta_i^{1/q})=s_i^{1/q}.
    \]
    Indeed, $s_i=\varphi_H(\zeta_i)\in Z(\widehat{H})$ and $\varphi_H(\Fr)\cdot \Fr^{-1}\in \widehat{H}$, so the first equality holds. Since $\Fr$ acts on $\widehat{\mathbb{S}}\cong X_\ast(\widehat{\mathbb{S}})\otimes \overline{k}^\times$ as $\Fr(x\otimes a)=\Fr(x)\otimes a^q$, as an element of $\widehat{\mathbb{S}}$ we have
    \[
        \Fr(s_i)=(s_i^{1/q})^q=s_i.
    \]
    Thus $s_i\in \widehat{\mathbb{S}}(k)$.
\end{proof}

Thus we obtain a character $\theta\colon \mathbb{S}(k)\to \C^\times.$ It is independent of the choice of $i$; indeed, if we change $i$ with another embedding $i'$, we have $n\in (\widehat{\Z}^{p})^\times=\varprojlim_{(n,p)=1}(\Z/n\Z)^\times$ such that $i'(x)=i(x)^n$. Then the character corresponding to $s_i$ changes into $\theta^n$. Also, we have $\zeta_{i}=\zeta_{i'}^n$ and $s_i=s_{i'}^n$. Thus the character corresponding to $s_{i'}$ is $\theta$. 

For $\xi\in H^1(\Gamma_F,H)$ and an unramified maximally split maximal torus $S'\subset H_\xi$, we can consider $\theta$ as a character of $\mathbb{S}'(k)$ as follows: for an unramified extension $E/F$ over which $H_\xi$ splits, we have an inner twist $f\colon H_E\to H_{\xi,E}$ such that $f(S)=S'$. Then $f$ induces an isomorphism $\mathbb{S}(\ell)\cong \mathbb{S}'(\ell)$, where $\ell$ is the residue field of $E$. Under this identification, $\Fr$ acts on $\mathbb{S}'(\ell)$ by $w\cdot \Fr$ for some $w\in W(S_E,H_E)$. Since $\theta$ is obtained from $s_i\in Z(\widehat{H})$, the character $\theta_\ell\coloneq \theta\circ N_{\ell/k}\colon \mathbb{S}(l)\to \C^\times$ does not change when conjugated by $w$. Then $\theta_\ell$ is $\Fr$-stable as a character of $\mathbb{S}'(\ell)$, so it factors through the norm map $N_{\ell/k}\colon \mathbb{S}'(\ell)\to \mathbb{S}'(k)$ as in \cite[5.3]{DL}. Thus $\theta_\ell$ is the inflation of a character $\theta=\theta_{\xi}$ of $\mathbb{S}'(k)$ along $N_{\ell/k}$.

Take a vertex $x\in \mathcal{A}(S',H_\xi)$. Via $\iota_\xi$, we can consider it as a vertex of $\mathcal{A}(S',G_{\overline{\xi}})$. Then we have the reductive quotient $\mathbb{H}\coloneq \mathbb{H}_{\xi,x},\ \mathbb{G}\coloneq \mathbb{G}_{\xi,x}$ of the parahoric subgroups corresponding to $x$, both of which contain $\mathbb{S}'$ as a maximal torus. Choosing an embedding $i\colon\overline{k}^\times\hookrightarrow \C^\times$, we obtain $s=s_i\in \widehat{\mathbb{S}}'(k)$ from the character $\theta$. Also, $i$ specifies a topological generator $\zeta=\zeta_i$ of $I_F/P_F$.

\begin{lemma}
    Under the above notation, $\widehat{\mathbb{H}}\cong Z_{\widehat{\mathbb{G}}}(s)^\circ$.
\end{lemma}

\begin{proof}
    Both sides contain the maximal torus $\widehat{\mathbb{S}}'$, so it suffices to show that their absolute root systems with respect to $\widehat{\mathbb{S}}'$ coincide. Thus we may extend the base field to $F^\ur$ and assume $\xi=1,\ S'=S$.

    Put
    \[
        \Psi_{S,G,x}\coloneq\{\psi\in \Psi_{S,G}\mid \psi(x)=0\}
    \]
    and let $\Phi_{S,G,x}$ be the derivation of it. We also define $\Psi_{S,H,x}$ and $\Phi_{S,H,x}$ similarly. According to \cite[Theorem 8.4.10]{Kaletha_Prasad_2023}, the root system of $\mathbb{H}$ and $\mathbb{G}$ are $\Phi_{S,H,x}$ and $\Phi_{S,G,x}$ respectively, under the identification $X^\ast(S)\cong X^\ast(\mathbb{S})$. Since $\Psi_{S,H}=\{\psi\in \Psi_{S,G}\mid \dot{\psi}\in \Phi_{S,H}\}$, we have $\Phi_{\mathbb{S},\mathbb{H}}=\Phi_{S,H}\cap \Phi_{\mathbb{S},\mathbb{G}}$. Thus $\Phi_{\widehat{\mathbb{S}},\widehat{\mathbb{H}}}=\Phi_{\widehat{S},\widehat{H}}\cap \Phi_{\widehat{\mathbb{S}},\widehat{\mathbb{G}}}$. Now we can express 
    \[
        \Phi_{\widehat{S},\widehat{H}}=\{\alpha\in \Phi_{\widehat{S},\widehat{G}}\mid \alpha(s)=1\}.
    \]
    Therefore we have 
    \[
        \Phi_{\widehat{\mathbb{S}},\widehat{\mathbb{H}}}=\{\alpha\in \Phi_{\mathbb{S},\mathbb{G}}\mid \alpha(s)=1\},
    \]
    and it coincides with the root system of $Z_{\widehat{\mathbb{G}}}(s)^\circ$.
\end{proof}

Let $\widetilde{\mathbb{H}}=H_{\xi,x}/H_{\xi,x,0+}$ and $\widetilde{\mathbb{G}}=G_{\overline{\xi},x}/G_{\overline{\xi},x,0+}$. According to \cite[(10)]{Opd}, we can describe the component groups $\pi_0(\widetilde{\mathbb{H}})$ and $\pi_0(\widetilde{\mathbb{G}})$ as follows: Fix a $\Fr$-stable chamber $\mathcal{C}\subset \mathcal{A}(S'_{F^\ur},H_{\xi,F^\ur})$ whose closure contains $x$. Since $W(\Psi_{S',H_\xi})$ acts on the set of chambers of $\mathcal{A}(S'_{F^\ur},H_{\xi,F^\ur})$ simply transitively, we have a $\Fr$-equivariant splitting
\[
    N_{H_\xi}(S')(F^\ur)/S'(F^\ur)_0\cong W(\Psi_{S',H_{\xi}})\rtimes \Omega_H,
\]
where $\Omega_H\subset N_{H_\xi}(S')(F^\ur)/S'(F^\ur)_0$ is the stabilizer of $\mathcal{C}$. Then 
\[
    \pi_0(\widetilde{\mathbb{H}})\cong\Omega_{H,x}\coloneq \{n\in\Omega_H\mid n\cdot x=x\}.
\]
Similarly, $\pi_0(\widetilde{\mathbb{G}})\cong\Omega_{G,x}$ holds. Also, we can restrict the canonical morphism $\Omega_H\to \Omega_G$ to $\Omega_{H,x}\to \Omega_{G,x}$.

\begin{lemma}\label{lem:isom_comp_H1}
    Let $\mathbb{H}^1$ be the inverse image of $\Ker(\Omega_{H,x}\to \Omega_{G,x})$ along the canonical map $\widetilde{\mathbb{H}}\to \pi_0(\widetilde{\mathbb{H}})\cong\Omega_{H,x}$. Then there exists an isomorphism $\pi_0(\mathbb{H}^1)\cong \pi_0(Z_{\widehat{\mathbb{G}}}(s))$ which is anti-equivariant under $\Fr$ and makes the following diagram commute:
    \[
        \begin{tikzcd}
            \pi_0(\mathbb{H}^1)\arrow[r,"\sim"]\arrow[d]&\pi_0(Z_{\widehat{\mathbb{G}}}(s))\arrow[d]\\
            \Out(\mathbb{H})\arrow[r,"\sim"]&\Out(\widehat{\mathbb{H}}),
        \end{tikzcd}
    \]
    where the vertical arrows are obtained from the adjoint actions on the identity component.
\end{lemma}

\begin{proof}
    Since $N_{H_{\xi,x}}(S')/S'_0\cong N_{\widetilde{\mathbb{H}}}(\mathbb{S}')/\mathbb{S}'$ and the same is true for $G_{\overline{\xi}}$, the following diagram commutes:
    \[
        \begin{tikzcd}
            N_{\widetilde{\mathbb{H}}}(\mathbb{S}')/\mathbb{S}'\arrow[r,hook]\arrow[d,two heads]&N_{\widetilde{\mathbb{G}}}(\mathbb{S}')/\mathbb{S}'\arrow[d,two heads]\\
            \Omega_{H,x}\arrow[r]&\Omega_{G,x}.
        \end{tikzcd}
    \]
    Now, the kernel of $N_{\widetilde{\mathbb{G}}}(\mathbb{S}')/\mathbb{S}'\to \Omega_{G,x}$ is $W(\mathbb{S}',\mathbb{G})$, so the inverse image of $\Ker(\Omega_{H,x}\to \Omega_{G,x})$ along the left vertical arrow is $N_{\widetilde{\mathbb{H}}}(\mathbb{S}')/\mathbb{S}'\cap W(\mathbb{S}',\mathbb{G})$, on which the derivation $N_{H_\xi}(S')(F^\ur)/S'(F^\ur)_0\twoheadrightarrow W(S',H_\xi)\subset W(S',G_{\overline{\xi}})$ is injective. Thus we can consider it as a subgroup of $W(S',G_{\overline{\xi}})$, and then we have
    \[
        N_{\widetilde{\mathbb{H}}}(\mathbb{S}')/\mathbb{S}'\cap W(\mathbb{S}',\mathbb{G})=W(S',H_\xi)\cap W(\mathbb{S}',\mathbb{G}).
    \]
    Forgetting the action of $\Fr$, we have
    \begin{equation}
        W(S',H_\xi)=W(S,H)\cong W(\widehat{S},\widehat{H})=\{w\in W(\widehat{S},\widehat{G})\mid w s w^{-1}=s\}.\label{eq:weyl_of_H}
    \end{equation}
    Thus $W(S,H_\xi)\cap W(\mathbb{S}',\mathbb{G})$ is identified with
    \[
        W(\widehat{\mathbb{S}}',\widehat{\mathbb{G}})_s=\{w\in W(\widehat{\mathbb{S}}',\widehat{\mathbb{G}})\mid wsw^{-1}=s\}
    \]
    via the isomorphism $W(\mathbb{S}',\mathbb{G})\cong W(\widehat{\mathbb{S}}',\widehat{\mathbb{G}})$, which is anti-equivariant under $\Fr$. Since any maximal torus of $Z_{\widehat{\mathbb{G}}}(s)^\circ$ is conjugate to $\widehat{\mathbb{S}}'$ under $Z_{\widehat{\mathbb{G}}}(s)^\circ(\overline{k})$, we have
    \[
        \pi_0(Z_{\widehat{\mathbb{G}}}(s))=N_{Z_{\widehat{\mathbb{G}}}(s)}(\widehat{\mathbb{S}}')/N_{Z_{\widehat{\mathbb{G}}}(s)^\circ}(\widehat{\mathbb{S}}')
        \cong W(\widehat{\mathbb{S}}',\widehat{\mathbb{G}})_s/W(\widehat{\mathbb{S}}',\widehat{\mathbb{H}}).
    \]
    Thus  we obtain an expected isomorphism $\pi_0(\mathbb{H}^1)\cong \pi_0(Z_{\widehat{\mathbb{G}}}(s))$. Also, comparing the action of $W(\mathbb{S}',\mathbb{G})=W(\widehat{\mathbb{S}}',\widehat{\mathbb{G}})$ on $\mathbb{S}'$ and $\widehat{\mathbb{S}}'$, we verify the commutativity of the above diagram.
\end{proof}

\begin{corollary}\label{cor:jord_G_H1}
    There exists a bijection
    \[
        J=J_{\mathbb{G},s}\colon \Irr(\mathbb{G}(k))_{s,\cusp}\to \Irr(\mathbb{H}^1(k))_{\unip,\cusp},
    \]
    where $\Irr(\mathbb{G}(k))_{s,\cusp}\subset \Irr(\mathbb{G}(k))_s$ is the subset of cuspidal representations, and $\Irr(\mathbb{H}^1(k))_{\unip,\cusp}\subset\Irr(\mathbb{H}^1(k))_{\unip}$ is defined similarly.
\end{corollary}

\begin{proof}
    The Jordan decomposition gives bijections
    \[
        \Irr(\mathbb{G}(k))_s\to \Irr(Z_{\widehat{\mathbb{G}}}(s)(k))_\unip,
    \]
    and 
    \[
        \Irr(Z_{\widehat{\mathbb{G}}}(s)^\circ(k))_\unip\to \Irr(\mathbb{H}(k))_\unip.
    \]
    Also, these bijections can be restricted to those between cuspidal representations because $Z(\mathbb{G})^\circ$ and $Z(\mathbb{H})^\circ$ have the same split rank (\cite[Theorem 3.2.22]{Geck_Malle_2020}). Then it suffices to show that the second bijection can be lifted to a bijection
    \begin{equation}
        \Irr(Z_{\widehat{\mathbb{G}}}(s)(k))_{\unip,\cusp}\to \Irr(\mathbb{H}^1(k))_{\unip,\cusp}.\label{eq:unip_discon}
    \end{equation}
    Now \cite[Lemma 15.7]{FOS} says that any unipotent cuspidal representation $\rho\in \Irr(\mathbb{H}(k))$ can be extended to an irreducible representation of $\widetilde{\mathbb{H}}(k)$, hence in particular that of $\mathbb{H}^1(k)$. The extendibility implies that $\pi_0(\mathbb{H}^1)(k)$ stabilizes $\rho$, so the corresponding representation $\widehat{\rho}\in \Irr(\widehat{\mathbb{H}}(k))_{\unip,\cusp}$ is stabilized by $\pi_0(Z_{\widehat{\mathbb{G}}}(s))(k)$. Then \cite[Proposition 11.5.3]{Digne_Michel_2020} shows that $\widehat{\rho}$ can be extended to an irreducible representation of $Z_{\widehat{\mathbb{G}}}(s)(k)$. Thus both sides of \eqref{eq:unip_discon} are $A_0^\vee$-torsor, where $A_0^\vee$ is the character group of $A_0=\pi_0(\mathbb{H}^1)(k)\cong \pi_0(Z_{\widehat{\mathbb{G}}}(s))(k)$, and we can construct a bijection as \eqref{eq:unip_discon} by fixing a base point in each side.
\end{proof}

\begin{proposition}\label{prop:jord_surj}
    There exists a surjection
    \[
        \widetilde{J}'_0\colon \Irr(\widetilde{\mathbb{H}}(k))_{\unip,\cusp}\twoheadrightarrow \Irr(\widetilde{\mathbb{G}}(k))_{s,\cusp},
    \]
    where $\Irr(\widetilde{\mathbb{G}}(k))_{s,\cusp}$ is the set of irreducible representations $\rho$ of $\widetilde{\mathbb{G}}(k)$ such that an irreducible component of $\rho|_{\mathbb{G}(k)}$ belongs to $\Irr(\mathbb{G}(k))_{s,\cusp}$. 
\end{proposition}

\begin{proof}
    For each $\rho_0\in \Irr(\mathbb{G}(k))_{s,\cusp}$, we denote by $\Irr(\widetilde{\mathbb{G}}(k))_{\rho_0}$ the set of $\rho\in \Irr(\widetilde{\mathbb{G}}(k))$ such that $\rho|_{\mathbb{G}(k)}$ contains $\rho_0$. Then
    \[
        \Irr(\widetilde{\mathbb{G}}(k))_{s,\cusp}=\bigcup_{\rho_0} \Irr(\widetilde{\mathbb{G}}(k))_{\rho_0}.
    \]
    We define $\Irr(\widetilde{\mathbb{H}}(k))_{\rho'_0}$ for $\rho'_0\in \Irr(\mathbb{H}^1(k))_{\unip,\cusp}$ similarly. We will construct a surjection
    \[
        \Irr(\widetilde{\mathbb{H}}(k))_{J(\rho_0)}\twoheadrightarrow \Irr(\widetilde{\mathbb{G}}(k))_{\rho_0}.
    \]
    Let $\widetilde{\mathbb{G}}(k)_s$ be the stabilizer of the Lusztig series $\Irr(\mathbb{G}(k))_s\subset \Irr(\mathbb{G}(k))$ in $\widetilde{\mathbb{G}}(k)$. Now the pair $(\mathbb{S}',\theta)$ is \emph{maximally split} in the sense of \cite[Definition 5.25]{DL}, so for any $g\in \widetilde{\mathbb{G}}(k)_s$, there exists $g_0\in \mathbb{G}(k)$ such that $g_0g$ stabilizes $(\mathbb{S}',\theta)$ (\cite[Proposition 5.26]{DL}). Thus we have
    \[
        \widetilde{\mathbb{G}}(k)_s=\mathbb{G}(k)\cdot N_{\widetilde{\mathbb{G}}}(\mathbb{S}')(k)_\theta.
    \]
    Also, we can deduce that
    \[
        N_{\widetilde{\mathbb{H}}}(\mathbb{S}')(k)/\mathbb{S}'(k)\cong N_{\widetilde{\mathbb{G}}}(\mathbb{S}')(k)_\theta/\mathbb{S}'(k)
    \]
    by using \eqref{eq:weyl_of_H}. Thus $\widetilde{\mathbb{G}}(k)_s/\mathbb{G}(k)\subset \pi_0(\widetilde{\mathbb{G}})(k)=\Omega_{G,x}^\Fr$ is precisely the image of $\Omega_{H,x}^\Fr\to \Omega_{G,x}^\Fr$, and we have a canonical isomorphism
    \begin{equation}
        \widetilde{\mathbb{H}}(k)/\mathbb{H}^1(k)\cong \widetilde{\mathbb{G}}(k)_s/\mathbb{G}(k).\label{eq:H_and_G_s}
    \end{equation}
    We denote this finite abelian group by $A$. Then its character group $A^\vee$ acts transitively on $\Irr(\widetilde{\mathbb{H}}(k))_{J(\rho_0)}$ and $\Irr(\widetilde{\mathbb{G}}(k)_s)_{\rho_0}$ by twisting characters. Also, the former action is simply transitive because $J(\rho_0)$ can be extended to an irreducible representation of $\widetilde{\mathbb{H}}(k)$ (\cite[Lemma 15.7]{FOS}). Now $\widetilde{\mathbb{G}}(k)_s$ contains the stabilizer of $\rho_0$ in $\widetilde{\mathbb{G}}(k)$, so the induction map
    \[
        \Irr(\widetilde{\mathbb{G}}(k)_s)_{\rho_0}\to \Irr(\widetilde{\mathbb{G}}(k))_{\rho_0};\quad \rho_1\mapsto \operatorname{Ind}_{\widetilde{\mathbb{G}}(k)_s}^{\widetilde{\mathbb{G}}(k)}\rho_1
    \]
    is bijective. Therefore we can construct a surjection
    \[
        \Irr(\widetilde{\mathbb{H}}(k))_{J(\rho_0)}\twoheadrightarrow \Irr(\widetilde{\mathbb{G}}(k))_{\rho_0}
    \]
    by fixing base points. Also, the extendibility of all $J(\rho_0)$ implies that 
    \[
        \Irr(\widetilde{\mathbb{H}}(k))_{\unip,\cusp}=\coprod_{\rho_0} \Irr(\widetilde{\mathbb{H}}(k))_{J(\rho_0)},
    \]
    so we have a surjection
    \[
        \widetilde{J}_0'\colon \Irr(\widetilde{\mathbb{H}}(k))_{\unip,\cusp}\twoheadrightarrow \Irr(\widetilde{\mathbb{G}}(k))_{s,\cusp}.
    \]
\end{proof}

\begin{corollary}\label{cor:tame_twist}
    There exists a surjection
    \[
        \widetilde{J}'=\widetilde{J}'_{\mathbb{G},s}\colon \Irr(\widetilde{\mathbb{H}}(k))_{s,\cusp}\twoheadrightarrow \Irr(\widetilde{\mathbb{G}}(k))_{s,\cusp}.
    \]
\end{corollary}

\begin{proof}
    It suffices to show that the character $\theta$ corresponding to $s$ can be extended to a character $\widetilde{\theta}$ of $H_\xi(F)$; indeed, if such a character $\widetilde{\theta}$ exists, we have a bijection
    \[
        \Irr(\widetilde{\mathbb{H}}(k))_{\unip,\cusp}\to \Irr(\widetilde{\mathbb{H}}(k))_{s,\cusp};\quad \rho\mapsto \widetilde{\theta}\otimes \rho.
    \]
    
    Take a finite unramified extension $E/F$ such that $S'_E$ splits. Then we can consider $\theta$ as a character of $S'(E)_0/S'(E)_{0+}$ by composing the norm map $N_{E/F}\colon S'(E)\to S'(F)$. Since we have
    \[
        \Phi_{\widehat{S},\widehat{H}}=\{\alpha\in \Phi_{\widehat{S},\widehat{G}}\mid \alpha(s)=1\},
    \]
    we can also express
    \[
        \Phi_{S',H_\xi}=\{\alpha\in \Phi_{S',G_{\overline{\xi}}}\mid \theta\circ N_{E/F}\circ \alpha^\vee|_{\mathcal{O}_E}=1\}.
    \]
    Let $H_{\xi,\scon}$ be the simply connected cover of the derived group $H_{\xi,\der}$ and $p\colon H_{\xi,\scon}\to H_\xi$ the canonical homomorphism. Then the commutator operation $[,]\colon H_\xi\times H_\xi\to H_\xi;\ (x,y)\mapsto xyx^{-1}y^{-1}$ factors through $p$. Indeed, let $\widetilde{x},\widetilde{y}$ be lifts of $x,y$ in $Z(H_\xi)\times H_{\xi,\scon}$. Then $\widetilde{x}\widetilde{y}\widetilde{x}^{-1}\widetilde{y}^{-1}\in H_{\xi,\scon}$ does not depend on the choice of $(\widetilde{x},\widetilde{y})$. Thus $(Z(H_\xi)\times H_{\xi,\scon})^2\to H_{\xi,\scon};\ (\widetilde{x},\widetilde{y})\mapsto \widetilde{x}\widetilde{y}\widetilde{x}^{-1}\widetilde{y}^{-1}$ factors as $(Z(H_\xi)\times H_{\xi,\scon})^2\to H_\xi\times H_\xi\to H_{\xi,\scon}$, and the composite $H_\xi\times H_\xi\to H_{\xi,\scon}\xrightarrow{p}H_\xi$ coincides with $[,]$. Then $p(H_{\xi,\scon}(F))\supset [H_\xi(F),H_\xi(F)]$, so $H_\xi(F)/p(H_{\xi,\scon}(F))$ is abelian. Also,
    \[
        p(H_{\xi,\scon}(F))\cap S'(F)_0=p(S'_\scon(F)_0),
    \]
    where $S'_\scon\subset H_{\xi,\scon}$ is the preimage of $S'$ along $p$. Consider a commutative diagram:
    \[
        \begin{tikzcd}
            S'_\scon(E)_0\arrow[r,"p"]\arrow[d,"N_{E/F}"]&S'(E)_0\arrow[d,"N_{E/F}"]\\
            S'_\scon(F)_0\arrow[r,"p"]&S'(F)_0.
        \end{tikzcd}
    \]
    Then the image of the top horizontal arrow is generated by $\alpha^\vee(\mathcal{O}_E)$, $\alpha\in\Phi_{S',H_\xi}$. Thus $\theta\circ N_{E/F}$ is trivial on $p(S'_\scon(E)_0)$. Since two vertical arrows are surjective, we have $\theta|_{p(S'_\scon(F)_0)}=1$. Therefore $\theta$ can be considered as a character on
    \[
        S'(F)_0/p(S'_\scon(F)_0)\subset H_\xi(F)/p(H_{\xi,\scon}(F)),
    \]
    and we can take some extension $\widetilde{\theta}\colon H_\xi(F)/p(H_{\xi,\scon}(F))\to\C^\times$.
\end{proof}

\begin{proposition}\label{prop:jord_bij}
    The surjection $\widetilde{J}_0'$ constructed in Proposition~\ref{prop:jord_surj} (hence also $\widetilde{J}'$) is bijective if and only if the following condition holds:
    \begin{enumerate}[label={$(\mathrm{B})$}]
        \item For any $\rho\in \Irr(\widetilde{\mathbb{G}}(k))_{s,\cusp}$ and $\rho_0\in \Irr(\mathbb{G}(k))_s$, the number of irreducible components of $\rho|_{\mathbb{G}(k)}$ which belong to $\mathbb{G}_\mathrm{ad}(k)\cdot \rho_0$ is at most one.
    \end{enumerate}
\end{proposition}

\begin{proof}
    
    By construction, $\widetilde{J}'_0$ is bijective if and only if $A^\vee$ acts on $\Irr(\widetilde{\mathbb{G}}(k)_s)_{\rho_0}$ simply transitively and 
    \[
        \Irr(\widetilde{\mathbb{G}}(k))_{s,\cusp}=\coprod_{\rho_0}\Irr(\widetilde{\mathbb{G}}(k))_{\rho_0}.
    \]
    We examine the action of $A=\widetilde{\mathbb{G}}(k)_s/\mathbb{G}(k)$ on $\Irr(\mathbb{G}(k))_s$ closely. Take an embedding $i\colon Z(\mathbb{G})\hookrightarrow \mathbb{T}$ into a torus over $k$ and define
    \[
        i^\diamond=(i\circ\Ad(a))_{a\in A}\colon Z(\mathbb{G})\hookrightarrow \mathbb{T}^\diamond\coloneq \mathbb{T}^A.
    \]
    Then $A$ acts naturally on $\mathbb{T}^\diamond$, and $i^\diamond$ is $A$-equivariant. We define $\mathbb{G}^\diamond$ as
    \[
        0\to Z(\mathbb{G})\xrightarrow{(z,i^\diamond(z))} \mathbb{G}\times \mathbb{T}^\diamond\to \mathbb{G}^\diamond\to 0.
    \]
    Then $\mathbb{G}^\diamond$ is a reductive group with connected center $\mathbb{T}^\diamond$, and $\widetilde{\mathbb{G}}(k)_s$ naturally acts on it. Also, the morphism $\mathbb{G}\hookrightarrow \mathbb{G}\times \mathbb{T}^\diamond\to \mathbb{G}^\diamond$ is a regular embedding (in the sense of \cite[7]{Lusztig_discon}) which is equivariant under the action of $\widetilde{\mathbb{G}}(k)_s$. Fixing a Borel pair and pinning of $\mathbb{G}^\diamond$, we can see that $A$ acts on $\mathbb{G}^\diamond$. Then $A$ also acts on its dual $\widehat{\mathbb{G}}^\diamond$. Take a lift $s^\diamond$ of $s$ along $p\colon \widehat{\mathbb{G}}^\diamond(k)\twoheadrightarrow \widehat{\mathbb{G}}(k)$ and put
    \[
        (\widehat{\mathbb{H}}^\diamond)^\sim\coloneq Z_{\widehat{\mathbb{G}}^\diamond\rtimes A}(s)=\{g\in \widehat{\mathbb{G}}^\diamond\rtimes A\mid p(\Ad(g)s^\diamond)=s\}.
    \]
    Then its identity component $\widehat{\mathbb{H}}^\diamond=Z_{\widehat{\mathbb{G}}^\diamond}(s^\diamond)$ is isogenous to $\widehat{\mathbb{H}}$. Also, since $A$ stabilizes the Lusztig series $\Irr(\mathbb{G}(k))_s$, it also stabilizes the $\widehat{\mathbb{G}}(k)$-conjugacy class of $s$ as a subgroup of $\mathrm{Out}(\widehat{\mathbb{G}}(k))$, so the composite $(\widehat{\mathbb{H}}^\diamond)^\sim(k)\hookrightarrow \widehat{\mathbb{G}}^\diamond(k)\rtimes A\twoheadrightarrow A$ is surjective (though not necessarily $A\subset (\widehat{\mathbb{H}}^\diamond)^\sim(k)$). Moreover, we have
    \begin{align}
        \frac{N_{\widetilde{\mathbb{H}}}(\mathbb{S}')(k)}{\mathbb{S}'(k)}&\cong \frac{N_{\widetilde{\mathbb{G}}(k)_s}(\mathbb{S}')_\theta}{\mathbb{S}'(k)}\\
        &\cong \frac{N_{\mathbb{G}^\diamond\rtimes A}(\mathbb{S}^\diamond)(k)_\theta}{\mathbb{S}^\diamond(k)}\\
        &\cong \frac{N_{\widehat{\mathbb{G}}^\diamond\rtimes A}(\widehat{\mathbb{S}}^\diamond)(k)_s}{\widehat{\mathbb{S}}^\diamond(k)}\\
        &=\frac{N_{(\widehat{\mathbb{H}}^\diamond)^\sim}(\widehat{\mathbb{S}}^\diamond)(k)}{\widehat{\mathbb{S}}^\diamond(k)},
    \end{align}
    where $\mathbb{S}^\diamond=\mathbb{S}'\cdot \mathbb{T}^\diamond\subset \mathbb{G}^\diamond$.
    Since $\mathbb{S}'\subset \mathbb{H}$ is a quasi-split maximal torus, so is $h\mathbb{S}'h^{-1}$ for any $h\in \widetilde{\mathbb{H}}(k)$. Thus $\mathbb{S}'$ and $h\mathbb{S}'h^{-1}$ are conjugate with each other under $\mathbb{H}(k)$, and we have
    \[
        \widetilde{\mathbb{H}}(k)/\mathbb{H}(k)\cong N_{\widetilde{\mathbb{H}}}(\mathbb{S}')(k)/N_{\mathbb{H}}(\mathbb{S}')(k)\cong (N_{\widetilde{\mathbb{H}}}(\mathbb{S}')(k)/\mathbb{S'}(k))/W(\mathbb{S}',\mathbb{H})^\Fr.
    \]
    Similarly we have $(\widehat{\mathbb{H}}^\diamond)^\sim(k)/\widehat{\mathbb{H}}^\diamond(k)\cong (N_{(\widehat{\mathbb{H}}^\diamond)^\sim}(\widehat{\mathbb{S}}^\diamond)(k)/\widehat{\mathbb{S}}^\diamond(k))/W(\widehat{\mathbb{S}}^\diamond,\widehat{\mathbb{H}}^\diamond)^\Fr$. Therefore we obtain a canonical isomorphism
    \[
        \widetilde{A}\coloneq \widetilde{\mathbb{H}}(k)/\mathbb{H}(k)\cong (\widehat{\mathbb{H}}^\diamond)^\sim(k)/\widehat{\mathbb{H}}^\diamond(k).
    \]
    Then the bijection
    \[
        \Irr(\mathbb{H}(k))_\unip\cong \Irr(\widehat{\mathbb{H}}(k))_\unip\cong \Irr(\widehat{\mathbb{H}}^\diamond(k))_\unip
    \]
    is $\widetilde{A}$-equivariant. Since any cuspidal unipotent representation of $\mathbb{H}(k)$ is stable under the action of $\widetilde{A}$, so are those of $\widehat{\mathbb{H}}^\diamond(k)$. For any $\rho^\diamond\in \Irr(\mathbb{G}^\diamond(k))_{s^\diamond,\cusp}$ and $a\in A$, we have
    \[
        J(a\cdot \rho^\diamond)=a\cdot J(\rho^\diamond)\in \Irr(Z_{\widehat{\mathbb{G}}}(as^\diamond a^{-1}))_{\unip,\cusp}
    \]
    by the equivariance of the Jordan decomposition $J$ for reductive groups with connected center (\cite[Theorem 3.1]{CS_equiv_conn}). As we said above that $(\widehat{\mathbb{H}}^\diamond)^\sim(k)\hookrightarrow \widehat{\mathbb{G}}^\diamond(k)\rtimes A\twoheadrightarrow A$ is surjective, there exists $g\in \widehat{\mathbb{G}}^\diamond(k)$ such that $ga\in (\widehat{\mathbb{H}}^\diamond)^\sim(k)$, then $ga\cdot J(\rho^\diamond)=J(\rho^\diamond)$ in $\Irr(\widehat{\mathbb{H}}^\diamond(k))_{\unip,\cusp}$. This means that $a\cdot \rho^\diamond=\chi\otimes \rho^\diamond$, where $\chi\colon \mathbb{G}^\diamond(k)\to \C^\times$ is the character which corresponds to $z=\Ad(ga)(s^\diamond){s^\diamond}^{-1}\in Z(\widehat{\mathbb{G}}^\diamond)(k)$. Since $p(z)=1$, the restriction of $\chi$ to $\mathbb{G}(k)$ is trivial and we have $\rho^\diamond|_{\mathbb{G}(k)}=a\cdot \rho^\diamond|_{\mathbb{G}(k)}$. Take $\rho^\diamond$ such that $\rho^\diamond|_{\mathbb{G}(k)}$ contains $\rho_0$. Then $a\cdot \rho_0$ is contained in $a\cdot\rho^\diamond|_{\mathbb{G}(k)}=\rho^\diamond|_{\mathbb{G}(k)}$, so $a\cdot \rho_0\in \mathbb{G}^\diamond(k)\cdot \rho_0=\mathbb{G}_\mathrm{ad}(k)\cdot \rho_0$. Therefore, any $\rho^1\in \Irr(\widetilde{\mathbb{G}}(k)_s)_{\rho_0}$ contains irreducible components only belonging to $\mathbb{G}_\mathrm{ad}(k)\cdot \rho_0$ as a representation of $\mathbb{G}(k)$, and the condition (B) implies that 
    \[
        \operatorname{Stab}_{\widetilde{\mathbb{G}}(k)}(\rho_0)=\widetilde{\mathbb{G}}(k)_s
    \]
    and that $\rho_0$ can be extended to an irreducible representation of $\widetilde{\mathbb{G}}(k)_s$. Then the two properties mentioned at the beginning of this proof hold and $\widetilde{J}'$ is bijective. The converse direction is now easy.
\end{proof}

\begin{theorem}\label{thm:general_LLC}
    Let $G'$ be an inner form of $G$ and $\Phi_{\mathrm{e}}(G')_{0,\cusp}$ the set of depth-zero discrete L-parameters for $G'$ with a cuspidal enhancement. Then there exists a surjective map
    \[
        \LLC\colon {\Phi_{\mathrm{e}}(G')_{0,\cusp}/\mathord{\sim}}\twoheadrightarrow \Irr(G'(F))_{0,\cusp}.
    \]
    Moreover, it is bijective if and only if the condition $(\mathrm{B})$ in Proposition~\ref{prop:jord_bij} holds for any $\widetilde{\mathbb{G}}=G'_x/G'_{x,0+}$ and semisimple $s\in \widehat{\mathbb{G}}(k)$, where $x\in \mathcal{B}(G')$ a vertex and $\mathbb{G}=G'_{x,0}/G'_{x,0+}$.
\end{theorem}

\begin{proof}
    For $\varphi\in \Phi(\Ldual{G})_{0,\disc}$, put 
    \[
        \Phi_{\mathrm{e}}(G')_{[\varphi],\cusp}=\{(\varphi',\epsilon)\in \Phi_{\mathrm{e}}(G')_{0,\cusp}\mid \varphi'\in [\varphi]\}.
    \]
    Also, we define 
    \[
        \Xi_{G'}=\{\xi\in H^1(\Gamma_F,H)\mid G_{\overline{\xi}}=G'\}.
    \]
    Then, for any $(\varphi',\epsilon)\in \Phi_{\mathrm{e}}(G')_{[\varphi],\cusp}$, the central character $\epsilon|_{Z(\widehat{H})^\Fr}$ specifies $\xi\in \Xi_{G'}$, which gives a partition
    \[
        \Phi_{\mathrm{e}}(G')_{[\varphi],\cusp}=\coprod_{\xi\in \Xi_{G'}}\Phi_{\mathrm{e}}(G')_{[\varphi],\xi,\cusp}.
    \]
    As in Section~\ref{sec:const_of_H}, we can construct $\Ldual{H}=\Ldual{H_{[\varphi]}}$ and choose an embedding $\Ldual{j}\colon \Ldual{H}\hookrightarrow \Ldual{G}$. Then we get a factorization $\varphi'=\Ldual{j}\circ\varphi'_H$ for $\varphi'\in [\varphi]$ and a bijection
    \[
        \Phi_{\mathrm{e}}(G')_{[\varphi],\xi,\cusp}\xrightarrow{\sim} \Phi_{\mathrm{e}}(H_\xi)_{[\varphi_H],\cusp}.
    \]
    By \cite[Theorem 2]{FOS}, we have a bijection
    \[
        {\Phi_{\mathrm{e}}(H_\xi)_{\ur,\cusp}/\mathord{\sim}}\xrightarrow{\sim} \Irr(H_\xi(F))_{\unip,\cusp},
    \]
    and as in Corollary~\ref{cor:tame_twist}, we can twist this bijection and obtain
    \[
        {\Phi_{\mathrm{e}}(H_\xi)_{s,\cusp}/\mathord{\sim}}\xrightarrow{\sim} \Irr(H_\xi(F))_{s,\cusp}.
    \]
    Here, $\Phi_\mathrm{e}(H_\xi)_{s,\cusp}=\{(\varphi,\epsilon)\in \Phi_\mathrm{e}(H_\xi)_{0,\cusp}\mid \varphi(\zeta)=s\}$ and $\Irr(H_\xi(F))_{s,\cusp}$ is the set of $\rho\in \Irr(H_\xi(F))$ such that $\rho=\rho'\otimes\widetilde{\theta}$, where $\widetilde{\theta}$ is an extension of $\theta$ and $\rho'\in \Irr(H_\xi(F))_{\unip,\cusp}$. Now take an unramified maximally split maximal torus $S'\subset H_\xi$. Then any vertex $x\in \mathcal{B}(H_\xi)$ is conjugate to a vertex in $\mathcal{A}(S',H_\xi)$ under $H_\xi(F)$. Also, each $\rho\in \Irr(H_\xi(F))_{s,\cusp}$ is obtained by compact induction of some cuspidal representation $\rho_0\in \Irr(H_\xi(F)_{x}/H_\xi(F)_{x,0+})_s$ for a vertex $x\in \mathcal{B}(H_\xi)$ determined uniquely up to $H_\xi(F)$-conjugacy. Thus we have a bijection
    \[
        \Irr(H_\xi(F))_{s,\cusp}\xrightarrow{\sim}\coprod_x \Irr(\widetilde{\mathbb{H}}_{\xi,x}(k))_{s,\cusp},
    \]
    where $x$ runs over the conjugacy classes of vertices in $\mathcal{A}(S',H_\xi)$ and $\widetilde{\mathbb{H}}_{\xi,x}=H_{\xi,x}/H_{\xi,x,0+}$.
    Therefore we have a map
    \begin{align}
        \LLC_{[\varphi],\xi}\colon \Phi_{\mathrm{e}}(G')_{[\varphi],\xi,\cusp}/\widehat{H}\text{-conj.}&\xrightarrow{\sim} \Phi_{\mathrm{e}}(H_\xi)_{[\varphi_H],\cusp}/\widehat{H}\text{-conj.}\\
        &\hookrightarrow \Phi_{\mathrm{e}}(H_\xi)_{s,\cusp}/\widehat{H}\text{-conj.}\\
        &\xrightarrow{\sim} \Irr(H_{\xi}(F))_{s,\cusp}\\
        &\xrightarrow{\sim} \coprod_x \Irr(\widetilde{\mathbb{H}}_{\xi,x}(k))_{s,\cusp}\\
        &\xrightarrow{\widetilde{J}'} \coprod_x \Irr(\widetilde{\mathbb{G}}'_x(k))_{s,\cusp}\\
        &\hookrightarrow \Irr(G'(F))_{0,\cusp},
    \end{align}
    where $\widetilde{\mathbb{G}}'_x=G'_x/G'_{x,0+}$. It is clear that, for $\varphi'=\Ad(g)\varphi,\ \xi'=\Ad(g)\xi$ and $g\in \widehat{G}$, we have $\LLC_{[\varphi'],\xi'}\circ\Ad(g)=\LLC_{[\varphi],\xi}$ for suitable choices of $\widetilde{J}'$. Thus, by considering all $[\varphi]$ and $\xi\in\Xi_{G'}$, we have a map
    \[
        \LLC\colon \Phi_{\mathrm{e}}(G')_{0,\cusp}/\widehat{G}\text{-conj.}\to \Irr(G'(F))_{0,\cusp}.
    \]
    
    We will show the surjectivity of the map $\LLC$. Let $x\in \mathcal{B}(G')$ be a vertex and put $\mathbb{G}=G'_{x,0}/G'_{x,0+}$. Take any cuspidal representation $\rho_0\in \Irr(\mathbb{G}(k))$. Then it determines the conjugacy class of $s\in \widehat{\mathbb{G}}(k)$ such that $\rho_0\in \Irr(\mathbb{G}(k))_s$. Put $\widehat{\mathbb{H}}=Z_{\widehat{\mathbb{G}}}(s)^\circ$ and take a quasi-split maximal torus $\widehat{\mathbb{S}}'\subset \widehat{\mathbb{H}}$. Then we have a maximal torus $\mathbb{S}'\subset \mathbb{G}$ and an unramified maximal torus $S'\subset G'$ such that $S'_0/S'_{0+}=\mathbb{S}'$, which are determined up to conjugacy. We define a root subsystem $\Phi_{S',H'}\subset \Phi_{S',G'}$ as
    \[
        \Phi_{S',H'}=\{\alpha\in \Phi_{S',G'}\mid \theta\circ N_{E/F}\circ \alpha^\vee|_{\mathcal{O}_E}=1\},
    \]
    where $\theta$ is the character of $\mathbb{S}'(k)$ corresponding to $s$ and $E/F$ is a finite unramified extension over which $S'$ splits. Then it is stable under $\Fr$ and defines a reductive group $H'=H'_{\rho_0}$ over $F$. Also, we have a reductive group $\mathbb{H}$ over $k$ as the dual of $\widehat{\mathbb{H}}$, and $\Phi_{\mathbb{S}',\mathbb{H}}\subset \Phi_{S',H'}$ under the identification $X^\ast(S')\cong X^\ast(\mathbb{S})$. As in \cite[Theorem 3.2.22]{Geck_Malle_2020}, $\Irr(\mathbb{G}(k))_s$ contains a cuspidal representation if and only if $Z(\widehat{\mathbb{H}})^\circ$ has the same split rank as that of $Z(\widehat{\mathbb{G}})^\circ$ and $\widehat{\mathbb{H}}(k)$ has a unipotent cuspidal representation. In particular, the split rank of $Z(H')^\circ$ and $Z(G')^\circ$ are the same. 

    Let $H=H_{\rho_0}$ be the quasi-split inner form of $H'$ and $f\colon H\to H'$ be an inner twist defined over $F^\ur$. We can choose $f$ such that $S=f^{-1}(S')$ is a quasi-split maximal torus of $H$. For an inner twist $f_G\colon G\to G'$ defined over $F^\ur$, we have an $F^\ur$-embedding $f_G^{-1}|_{S'}\circ f|_S\colon S\to S'\to G$. Since the inner twist $f_G$ is unique up to conjugacy, we have a stable conjugacy class of embeddings $S\hookrightarrow G$. By Lemma~\ref{lem:embed_torus}, we obtain an embedding $S\hookrightarrow G$ defined over $F$.
    As Proposition~\ref{prop:embd_apart} and Corollary~\ref{cor:corr_weyl}, we have embeddings $\iota\colon \widetilde{\mathcal{A}}(S_{F^\ur},H_{F^\ur})\to \mathcal{A}(S_{F^\ur},G_{F^\ur})$ and $N_H(S)(F^\ur)/S(F^\ur)_0\to N_G(S)(F^\ur)/S(F^\ur)_0$. 
    We consider the cocycle 
    \[
        c\colon \Gamma_{F^\ur/F}\to N_G(S)(F^\ur)/S(F^\ur)_0;\quad \gamma\mapsto f^{-1}_G\circ \gamma(f_G)\mod{S(F^\ur)_0}.
    \]
    Since the action of $f^{-1}_G\circ \gamma(f_G)$ on $S(F^\ur)$ is the same as that of $f^{-1}\circ \gamma(f)$, the image of $c$ is contained in $N_H(S)(F^\ur)/S(F^\ur)_0$. Also, the embedding $S\hookrightarrow G$ determines $f_G$ up to $S(F^\ur)$-conjugacy, so the cohomology class of $c$ in $H^1(\Gamma_{F^\ur/F},N_H(S)(F^\ur)/S(F^\ur)_0)$ is well-defined. Then we obtain $\xi\in H^1(\Gamma_{F^\ur/F},\Omega_H)\cong H^1(\Gamma_F,H)$ via the morphism $N_H(S)(F^\ur)/S(F^\ur)_0\to \Omega_H$. Moreover, we have $H_\xi\cong H'$ by construction.
    Then we also have $\iota_\xi\colon \mathcal{A}(S',H')\hookrightarrow \mathcal{A}(S',G')$ and $x\in \mathcal{A}(S',G')$ can be considered as a point of $\mathcal{A}(S',H')$. Since $H'_{x,0}/H'_{x,0+}\cong \mathbb{H}$ and its connected center has the same split rank as that of $Z(\mathbb{G})^\circ$ (hence that of $Z(H')^\circ$), $x$ is a vertex of $\mathcal{A}(S',H')$. Thus we can take a unipotent supercuspidal representation $\pi_\unip\in \Irr(H'(F))$ which contains $J(\rho_0)\in \Irr(\mathbb{H}^1(k))_{\unip,\cusp}$. Via the local Langlands correspondence for unipotent supercuspidal representations, we obtain $(\varphi_H,\epsilon)=(\varphi_{H,\rho_0},\epsilon_{\rho_0})\in \Phi_\mathrm{e}(H')_{\xi,\ur,\cusp}$ from $\pi_\unip$. Since $\pi_\unip$ is expressed as the compact induction of an extension of $J(\rho_0)$ to an irreducible representation of $H'(F)_x$, it is determined up to \emph{weakly unramified twist} in the sense of \cite[p.~1139--40]{FOS}. Thus $[\varphi_H]$ and $\epsilon$ is determined up to conjugacy.

    On the other hand, the embedding $S\hookrightarrow G$ specifies a $\Fr$-stable $\widehat{G}$-conjugacy class of embeddings $\widehat{S}\hookrightarrow\widehat{G}$. Considering $\theta$ as a character of $\mathbb{S}(k)=S(F)_0/S(F)_{0+}$, we obtain an element $s\in \widehat{\mathbb{S}}\hookrightarrow \widehat{S}$ and we have $\widehat{H}=Z_{\widehat{G}}(s)\subset \widehat{G}$. As in Section~\ref{sec:const_of_H}, we extend it to an L-embedding $\Ldual{H}\hookrightarrow\Ldual{G}$, and (a twist of) the enhanced cuspidal L-parameter $(\varphi_H,\epsilon)$ gives $(\varphi,\epsilon)=(\varphi_{\rho_0},\epsilon_{\rho_0})\in \Phi_{\mathrm{e}}(G')_{0,\cusp}$. Then it is easy by construction that any $\rho\in \Irr(G'(F))_{0,\cusp}$ which contains $\rho_0$ is obtained from some L-parameter $\varphi'\in [\varphi]$ via $\LLC_{[\varphi],\xi}$. Thus $\LLC$ is surjective. 

    Now we consider the bijectivity. We claim that $\LLC$ is bijective if and only if $\LLC_{[\varphi],\xi}$ is bijective for any $[\varphi]$ and $\xi$. Suppose that $\pi=\LLC(\varphi,\epsilon)$. Then, $\pi$ determines the conjugacy class of $x\in \mathcal{B}(G')$ and $\rho\in \Irr(G'(F)_x/G'(F)_{x,0+})$ such that $\pi=\cind_{G'(F)_x}^{G'(F)}\rho$. 
    Let $\rho_0\in \Irr(G'(F)_{x,0}/G'(F)_{x,0+})$ be an irreducible component of $\rho|_{G'(F)_{x,0}}$, which is a cuspidal representation. Then the conjugacy class of $(x,\rho_0)$ does not depend on the choice of $\rho_0$. On the other hand, $([\varphi],\epsilon)$ defines $H_{[\varphi]},\ \iota_\xi\colon \mathcal{A}(S',H_{[\varphi],\xi})\to \mathcal{A}(S',G')$ and a vertex $y\in \mathcal{A}(S',G')$. Since $\LLC(\varphi,\epsilon)=\pi$, we may assume $x=y$. Then we have $H=H_{[\varphi]}\cong H_{\rho_0}$ and $H_{\xi}\cong H'_{\rho_0}$ by comparing the root systems of them. Also, $\xi=\xi_{\rho_0}$ under this identification. Then both of $(\varphi_H,\epsilon),\ (\varphi_{H,\rho_0},\epsilon_{\rho_0})\in \Phi_\mathrm{e}(H)_{\xi,\ur,\cusp}$ correspond to unipotent supercuspidal representations which contain $J(\rho_0)$. Since such representations are the same up to weakly unramified twists, we have $([\varphi_H],\epsilon)=([\varphi_{H,\rho_0}],\epsilon_{\rho_0})$, hence $([\varphi],\epsilon)=([\varphi_{\rho_0}],\epsilon_{\rho_0})$, up to conjugacy. In particular, $\#\LLC^{-1}(\pi)>1$ if and only if $\#\LLC_{[\varphi_{\rho_0}],\xi_{\rho_0}}^{-1}(\pi)>1$, so we have shown the claim. Varying $\varphi_H$, we can make the map
    \[
        \coprod_{[\varphi_H]} \Phi_{\mathrm{e}}(H_\xi)_{[\varphi_H],\cusp}\to \Phi_{\mathrm{e}}(H_\xi)_{s,\cusp}
    \]
    bijective, so the bijectivity of all $\LLC_{[\varphi],\xi}$ is equivalent to that of all $\widetilde{J}'=\widetilde{J}'_{\mathbb{G}'_x,s}$. Thus we complete the proof.
\end{proof}

\begin{remark}\label{rem:non_can_choices}
    As we have seen, there are several choices to construct the map $\LLC$. We explain how these choices affect it. 

    We first consider the choice of a hyperspecial vertex in Proposition~\ref{prop:embd_apart}. Such a vertex provides a reductive model over the ring of integers $\mathcal{O}_F$. According to \cite[Section 7]{Hales}, we can define a canonical normalization of the transfer factors from it, which should specify the bijection in Conjecture~\ref{conj:LLC}. Namely, this choice arises from the ambiguity of the local Langlands correspondence itself.

    We next consider the choice of L-embeddings in Section~\ref{sec:const_of_H} and of bijections of irreducible representations in Proposition~\ref{prop:jord_surj}. Both of them have ambiguity with respect to the weakly unramified twist. It is not clear whether we can choose them canonically, but at least, the method for regular supercuspidal representations used in \cite[p.~828]{DR} does not work in general. Rather, it seems reasonable that the choice of an L-embedding specifies a bijection between the irreducible representations by imposing some character relation just like endoscopy.
\end{remark}

\begin{remark}\label{rem:compari_to_DR}
    Suppose that $(\varphi,\epsilon)\in \Phi_\mathrm{e}(G')_{0,\cusp}$ satisfies that $\varphi|_{\SL_2(\C)}$ is trivial. Then the cuspidality induces that $\widehat{H}=\Z_{\widehat{G}}(\varphi(I_F))$ is a torus, and $\varphi$ is \emph{tame regular semisimple} in the sense of \cite[4.1]{DR}. In this case, any $\varphi'\in [\varphi]$ is conjugate to $\varphi$ by \cite[(16)]{DR}. Also, $H=S$ is compact, so $\widetilde{\mathbb{H}}(k)=\mathbb{H}(k)$. The enhancement $\epsilon$ specifies an embedding $S\hookrightarrow G'$, and the character $\theta\colon \mathbb{H}(k)=\mathbb{S}(k)\to \C^\times$ determines an irreducible cuspidal representation $\rho_0$ of $\widetilde{\mathbb{G}}(k)_s=\mathbb{G}(k)$. Therefore $\LLC(\varphi,\epsilon)=\cind(\rho_0)$ with no ambiguity, and this correspondence is the same one as in \cite[4.6]{DR}
\end{remark}

\section{The action of the fundamental group on parahoric subgroups}\label{sect:action_fund_parahoric}

Let $G$ be an unramified adjoint simple group over $F$. For each $\Fr$-stable facet $\mathcal{F}\subset \mathcal{B}(G_{F^\ur})$, we can consider the parahoric subgroup $G_{\mathcal{F},0}$. Put $\mathbb{G}=G_{\mathcal{F},0}/G_{\mathcal{F},0+}$ and $\widetilde{\mathbb{G}}=G_\mathcal{F}/G_{\mathcal{F},0+}$. We have seen (for a vertex, but it holds in general) that $\widetilde{\mathbb{G}}/\mathbb{G}\cong \Omega_{G,\mathcal{F}}$ and it acts on $\mathbb{G}$ by conjugation. In this section, we prove the following:

\begin{theorem}\label{thm:fund_preserv_pinning}
    There exists a $\Fr$-stable pinning $\{X_\alpha\}$ of $\mathbb{G}$ which satisfies the following: 
    For any $\omega\in \Omega_{G,\mathcal{F}}^\Fr$, there exists a lift $g_\omega\in G(F)_\mathcal{F}$ of $\omega$ which stabilizes $\{X_\alpha\}$.
\end{theorem}

Let $B\subset G$ be a Borel subgroup and $T\subset B$ a maximal torus. Then there exists a $\Fr$-stable pinning $\{X_\alpha\in \Lie U_\alpha(F^\ur)\}_{\alpha\in \Delta}$ for the Borel pair $(B,T)$, where $\Delta\subset \Phi_{T,G}$ is the root basis corresponding to $B$ and $U_\alpha\subset G$ is the root subgroup. Now $(B,T,\{X_\alpha\})$ defines a $\Fr$-stable chamber $\mathcal{C}\subset \mathcal{A}(T_{F^\ur},G_{F^\ur})\subset \mathcal{B}(G_{F^\ur})$ with a hyperspecial vertex $x_0\in \mathcal{B}(G)$ contained in the closure of it. Also, $\mathcal{C}$ defines a basis $\Delta_\aff$ of the affine root system $\Psi_{T,G}$. We put $\widetilde{\Delta}=\{\dot{\psi}\mid \psi\in \Delta_\aff\}$, then it coincides with $\Delta\cup\{-\theta\}$, where $\theta\in \Phi_{T,G}^+$ is the highest root. Without loss of generality, we may assume that the facet $\mathcal{F}$ is contained in the closure of $\mathcal{C}$. Then the root system $\Phi_{\mathbb{T},\mathbb{G}}$ of $\mathbb{G}$ is expressed as the derivation of
\[
    \Psi_{T,G,\mathcal{F}}=\{\psi\in \Psi_{T,G}\mid \psi|_\mathcal{F}=0\}.
\]
Also, a basis $\Delta_\mathcal{F}$ of $\Phi_{\mathbb{T},\mathbb{G}}$ is obtained as
\[
    \Delta_\mathcal{F}=\{\dot{\psi}\in \widetilde{\Delta}\mid \psi|_{\mathcal{F}}=0\}\subset \widetilde{\Delta}.
\]
For $\alpha\in \Phi_{T,G}$, $X_\alpha\in \Lie(U_\alpha)$ defines an isomorphism $\varphi_\alpha\colon \mathbb{G}_\mathrm{a}\to U_\alpha$ defined over $F^\ur$ such that $d\varphi_\alpha=X_\alpha$. For $c\in F^\ur$, we denote $\varphi_\alpha(c)$ by $cX_\alpha$ for simplicity. For each $\alpha\in \Phi_{T,G}$, a point $x\in\mathcal{F}$ defines a real number $N_\alpha$ such that
\[
    U_\alpha(F^\ur)\cap G(F^\ur)_{x,0}=\{cX_\alpha\mid \nu_F(c)\geq N_\alpha\}
\]
and
\[
    U_\alpha(F^\ur)\cap G(F^\ur)_{x,0+}=\{cX_\alpha\mid \nu_F(c)> N_\alpha\},
\]
where $\nu_F\colon F\twoheadrightarrow \Z\cup\{\infty\}$ is the valuation of $F$.
Also, for $\alpha\in \Phi_{\mathbb{T},\mathbb{G}}$, $N_\alpha$ is an integer and does not depend on the choice of $x\in\mathcal{F}$. If $\mathcal{F}=\{x_0\}$, then $\Delta_{x_0}=\Delta$ and $N_\alpha=0$ for $\alpha\in \Delta$. Thus $\{X_\alpha\}$ can be seen as a $\Fr$-stable pinning of $\mathbb{G}_0=G_{x_0,0}/G_{x_0,0+}$.

As a subgroup of $N_G(T)(F^\ur)/T(F^\ur)_0$, we consider that $\Omega_G$ is the stabilizer of the chamber $\mathcal{C}$. Then $\Omega_G$ permutes the elements of $\Delta_\aff$, which is an action on the affine Dynkin diagram with vertices $\Delta_\aff$. Thus each $\omega\in \Omega_{G,\mathcal{F}}$ stabilizes the Borel pair $(\mathbb{B}_\mathcal{F},\mathbb{T})$ of $\mathbb{G}$, where $\mathbb{B}_\mathcal{F}\subset \mathbb{G}$ is the Borel subgroup defined from the $\Fr$-stable root basis $\Delta_\mathcal{F}$. Then it suffices to find a suitable pinning for $(\mathbb{B}_\mathcal{F},\mathbb{T})$  and a lift $n_\omega\in N_G(T)(F^\ur)_\mathcal{F}$ which stabilizes this pinning for $\omega\in \Omega_{G,\mathcal{F}}^\Fr$.

\begin{proposition}\label{prop:preserv_for_standard}
    If $-\theta\notin\Delta_\mathcal{F}$, Theorem~\ref{thm:fund_preserv_pinning} holds for $\mathcal{F}$.
\end{proposition}

\begin{proof}
    From the $\Fr$-stable pinning $\{X_\alpha\}$ for $(B,T)$, we can construct a $\Fr$-equivariant lift
    \[
        n\colon W(T,G)\to N_G(T)
    \]
    as in \cite[(2.1)]{LS}. Then for any $w\in W(T,G)$ and $\alpha,\beta\in\Delta$ such that $\beta=w\cdot \alpha$, we have $n(w)\cdot X_\alpha=X_\beta$ (\cite[Proposition 9.3.5]{Springer_98}).
    
    Since $\Delta_\mathcal{F}\subset\Delta_{x_0}$, $\{X_\alpha\}_{\alpha\in\Delta_\mathcal{F}}$ can be considered as a pinning of $\mathbb{G}$. For each $\omega\in \Omega_{G,\mathcal{F}}^\Fr\subset N_G(T)(F)/T(F)_0$, we denote by $\overline{\omega}$ the image of $\omega$ in $W(T,G)$. Since it is fixed by $\Fr$, $\widetilde{\omega}\coloneq n(\overline{\omega})$ belongs to $N_G(T)(F)$. Also, there exists  $t\in T(F)$ such that $t\widetilde{\omega}\in G(F)_\mathcal{F}$. Since $G$ is adjoint, $X_\ast(T)$ has a basis consisting of the fundamental coweights $w_1,\dots,w_r$. We express that $t=w_1(a_1)\cdots w_r(a_r),\ a_i\in F^\ur$. We put
    \[
        a'_i=\varpi^{\nu_F(a_i)}
    \] 
    and $t'=w_1(a'_1)\cdots w_r(a'_r)$, where $\varpi\in F$ is a uniformizer. Then $t'\in T(F)$ and $t't^{-1}\in T(F)_0$. Thus $t'\widetilde{\omega}\in G(F)_\mathcal{F}$. Since $\widetilde{\omega}\cdot X_\alpha=X_{\overline{\omega}\cdot\alpha}$, we have
    \[
        t'\widetilde{\omega}\cdot X_\alpha=\varpi^m X_{\overline{\omega}\cdot\alpha},
    \]
    for some $m\in\Z$, but $t'\widetilde{\omega}\in G(F)_\mathcal{F}$ implies that $m$ must be zero for $\alpha\in\Delta_\mathcal{F}$. Thus $t'\widetilde{\omega}$ preserves a pinning of $\mathbb{G}$.
\end{proof}

If $G=\PGL_m$, then any vertex is hyperspecial and we may assume that $-\theta\notin\Delta_\mathcal{F}$. Thus Theorem~\ref{thm:fund_preserv_pinning} holds in this case. In the following, we assume that $G\neq \PGL_m$ and $\Delta_\mathcal{F}$ contains $-\theta\in \widetilde{\Delta}$.

Suppose $G=\PU_{2m+1}$ is an odd projective unitary group. Then $\Omega_G^\Fr$ is trivial (see \cite[p.~45]{Kaletha_Prasad_2023}), and the theorem also holds in this case. Hence we additionally assume that $G\neq \PU_{2m+1}$.

\begin{lemma}\label{lem:highest_root_indep}
    Fix $\alpha\in \Delta_\mathrm{long}$ and take $w\in W(T,G)$ such that $w\cdot \alpha=-\theta$. Then $n(w)\cdot X_\alpha\in \Lie U_{-\theta}(F^\ur)$ does not depend on the choice of $w$.
\end{lemma}

\begin{proof}
    According to \cite[Proposition 7.1]{CP_weyl}, there exists a unique element $y_\alpha\in W(T,G)$ such that $y_\alpha\cdot \alpha=\theta$, and that, for any $v\in W(T,G)$ with $v\cdot \alpha=\theta$, we can express $v=v'y_\alpha$, $v'\in \Stab_{W(T,G)}(\theta)$ which satisfies $\ell(v)=\ell(v')+\ell(y_\alpha)$, where $\ell$ is the length function on $W(T,G)$. Let $w_0\in W(T,G)$ be the longest element and $y'_\alpha=w_0y_\alpha$. Then for any $w\in W(T,G)$ with $w\cdot \alpha=-\theta$, we have $v\in \Stab_{W(T,G)}(\theta)$ such that $y'_\alpha=vw$ and $\ell(y'_\alpha)=\ell(v)+\ell(w)$, so $n(y'_\alpha)=n(v)n(w)$. Thus it suffices to show that $n(v)$ acts on $U_{-\theta}$ trivially for any $v\in \Stab_{W(T,G)}(\theta)$. It is equivalent to show that $n(v)$ acts trivially on $U_\theta$. Since $\theta$ is dominant, $\Stab_{W(T,G)}(\theta)$ is generated by simple reflections $s_{\alpha_i}$ for $\alpha_i\in\Delta$ orthogonal to $\theta$. Then we only need to show that $n(s_{\alpha_i})$ acts on $U_\theta$ trivially. Considering the semisimple subgroup of rank two generated by $U_{\pm\alpha_i}$ and $U_{\pm\theta}$, we can reduce to the case where $G$ is of type $A_1\times A_1,\ A_2,\ B_2$ or $G_2$. Then calculations in \cite[33.3--5]{Humphreys_75} verifies the claim.
\end{proof}

\begin{lemma}\label{lem:y_vs_y'}
    Let $w_0,y_\alpha,y'_\alpha\in W(T,G)$ be the elements as in the proof of Lemma~\ref{lem:highest_root_indep}. Then $n(y'_\alpha)\cdot X_\alpha=n(w_0)^{-1}n(y_\alpha)\cdot X_\alpha$.
\end{lemma}

\begin{proof}
    Since $y'_\alpha=w_0y_\alpha$ and $\ell(y'_\alpha)=\ell(w_0)-\ell(y_\alpha)=\ell(w_0)-\ell(y_\alpha^{-1})$, we have $n(w_0)=n(y'_\alpha)n(y_\alpha^{-1})$, so $n(y'_\alpha)=n(w_0)n(y_\alpha^{-1})^{-1}$. Then we need to show that 
    \[
        n(w_0)n(y_\alpha^{-1})^{-1}\cdot X_\alpha=n(w_0)^{-1}n(y_\alpha)\cdot X_\alpha.
    \]
    we put $X_\theta=n(y_\alpha^{-1})^{-1}X_\alpha$, then the above equation is equivalent to 
    \[
        n(w_0)^2\cdot X_\theta=n(y_\alpha)n(y_\alpha^{-1})\cdot X_\theta.
    \]
    According to \cite[Lemma 2.1.A]{LS}, we have
    \[
        n(w_0)^2=\prod_{\gamma>0}\gamma^\vee(-1),\quad n(y_\alpha)n(y_\alpha^{-1})=\prod_{\substack{\gamma>0,\\
        y_\alpha^{-1}\cdot\gamma<0}} \gamma^\vee(-1).
    \]
    Let $\rho^\vee\coloneq \sum_{\gamma>0}\gamma^\vee/2$. Then $\langle \theta,\rho^\vee\rangle\in \Z$, so we have
    \[
        n(w_0)^2\cdot X_\theta=(-1)^{\langle \theta,2\rho^\vee\rangle}X_\theta=X_\theta.
    \]
    On the other hand, \cite[Proposition 7.1]{CP_weyl} says that 
    \[
        N(y_\alpha^{-1})\coloneq \{\gamma\in \Phi_{T,G}^+\mid y_\alpha\cdot\gamma<0\}=\{\gamma\in \Phi_{T,G}^+\mid \langle\gamma,\alpha^{\vee}\rangle=-1\},
    \]
    and that, in the reduced expression $y_\alpha=s_{\beta_1}\cdots s_{\beta_\ell}$ of $y_\alpha$, each $s_\beta$ appears $m_\beta^\vee=m_\beta\cdot (\beta,\beta)/(\theta,\theta)$ times if $\beta\neq \alpha$, and $1+m_\beta^\vee$ times if $\beta=\alpha$, where $m_\beta$ is the integer defined as
    \[
        \theta=\sum_{\beta\in\Delta}m_\beta \cdot \beta.
    \]
    Now we have $N(y_\alpha)=-y_\alpha\cdot N(y_\alpha^{-1})$. Also, 
    \[
        N(y_\alpha^{-1})=\{\beta_\ell,s_{\beta_\ell}(\beta_{\ell-1}),\dots,s_{\beta_\ell}\cdots s_{\beta_2}(\beta_1)\},
    \]
    so the number of long roots and short roots inside it are $1+\sum_{\beta\in\Delta_\mathrm{long}} m_\beta^\vee$ and $\sum_{\beta\in\Delta_\mathrm{short}} m_\beta^\vee$, respectively. Therefore we have
    \begin{align}
        \sum_{\gamma\in N(y_\alpha)}\langle \theta,\gamma^\vee\rangle
        &=-\sum_{\substack{\gamma>0,\\\langle \gamma,\alpha^\vee\rangle=-1}}\langle \alpha,\gamma^\vee\rangle\\
        &=-\sum_{\substack{\gamma>0,\\\langle\gamma,\alpha^\vee\rangle=-1,\\
        \gamma\colon\text{long}}}\langle \alpha,\gamma^\vee\rangle-\sum_{\substack{\gamma>0,\\\langle\gamma,\alpha^\vee\rangle=-1,\\
        \gamma\colon\text{short}}}\langle \alpha,\gamma^\vee\rangle\\
        &=1+\sum_{\beta\in\Delta_\mathrm{long}} m_\beta+\sum_{\beta\in\Delta_{\mathrm{short}}} m_\beta\cdot \frac{(\beta,\beta)}{(\theta,\theta)}\cdot \frac{(\alpha,\alpha)}{(\beta,\beta)}\\
        &=1+\sum_{\beta\in \Delta} m_\beta,
    \end{align}
    which is the Coxeter number of $\Phi_{T,G}$. Since we assume that $G\neq \PGL_n$ nor $\PU_{2n+1}$, the Coxeter number must be even, so
    \[
        n(y_\alpha)n(y_\alpha^{-1})\cdot X_\theta=X_\theta.
    \] 
    Thus we complete the proof.
\end{proof}

\begin{lemma}\label{lem:dist_of_root}
    Let $\alpha,\beta\in\Delta$ be long roots such that $\langle \alpha,\beta^{\vee}\rangle=-1$. Then 
    \[
        n(y_\beta)X_\beta=-n(y_\alpha)X_\alpha.
    \]
\end{lemma}

\begin{proof}
    Put $r=s_\alpha s_\beta$. Then $r\cdot \alpha=\beta$, so there exists $x\in \Stab_{W(T,G)}(\theta)$ such that $y_\beta r=xy_\alpha$ and $\ell(xy_\alpha)=\ell(x)+\ell(y_\alpha)$. Since $n(x)$ acts on $U_\theta$ trivially, we have
    \begin{align}
        n(y_\alpha)\cdot X_\alpha&=n(x)n(y_\alpha)\cdot X_\alpha\\
        &=n(xy_\alpha)\cdot X_\alpha\\
        &=n(y_\beta r)\cdot X_\alpha\\
        &=n(y_\beta r)n(r)^{-1}\cdot X_\beta\\
        &=n(y_\beta r)n(r)^{-1}n(y_\beta)^{-1}n(y_\beta)\cdot X_\beta.
    \end{align}
    As in \cite[Lemma 2.1.A]{LS}, we have
    \[
        n(y_\beta)n(r)n(y_\beta r)^{-1}=\prod_{\substack{\gamma>0,\\ y_\beta^{-1}\cdot\gamma<0\\ r^{-1}y_\beta^{-1}\cdot\gamma>0}} \gamma^\vee(-1).
    \]
    Now we consider the set 
    \[
        \{\gamma>0\mid y_\beta^{-1}\cdot \gamma<0,r^{-1}y_\beta^{-1}\cdot \gamma>0\}.
    \]
    Since $y_\beta^{-1}\cdot \gamma<0$ and $r^{-1}y_\beta^{-1}\cdot \gamma>0$, we have $-y_\beta^{-1}\cdot \gamma=\alpha$ or $\alpha+\beta$. Also, $y_\beta\cdot (-y_\beta^{-1}\cdot\gamma)<0$ and $N(y_\beta^{-1})=\{\delta\in \Phi_{T,G}^+\mid \langle \delta,\beta^\vee\rangle=-1\}$, so $\langle -y_\beta^{-1}\cdot\gamma, \beta^\vee\rangle=-1$. Thus the above set is the singleton $\{-y_\beta\cdot \alpha\}$ and 
    \begin{align}
        n(y_\beta)\cdot X_\beta&=n(y_\beta)n(r)n(y_\beta r)^{-1}n(y_\alpha)\cdot X_\alpha\\
        &=(-1)^{\langle \theta,-y_\beta\cdot \alpha^\vee\rangle} n(y_\alpha)\cdot X_\alpha\\
        &=(-1)^{\langle \beta,-\alpha^\vee\rangle} n(y_\alpha)\cdot X_\alpha\\
        &=-n(y_\alpha)\cdot X_\alpha.
    \end{align}
\end{proof}

\begin{proof}[Proof of Theorem~\ref{thm:fund_preserv_pinning}]
    We first concern the case where $G$ is not of type $D_{2n+1}$. Take a simple root $\alpha\in\Delta$ such that $m_\alpha=1$ (such a simple root exists unless $G=E_8,F_4$ or $G_2$, but in these cases, $\Omega_G=\{1\}$ and there is nothing to prove). We define $X_{-\theta}\coloneq n(y'_\alpha)\cdot X_\alpha$. Examining the extended Dynkin diagram, we see that there exists a sequence of long simple roots $\alpha=\alpha_0,\alpha_1,\dots,\alpha_{2m}=\Fr\cdot\alpha$ such that $\langle \alpha_i,\alpha_{i+1}^\vee\rangle=-1$. Then Lemmas~\ref{lem:highest_root_indep},~\ref{lem:y_vs_y'}~and~\ref{lem:dist_of_root} deduce that $X_{-\theta}$ is stable under $\Fr$.

    Put $X'_\beta=\varpi^{N_\beta}X_\beta$ for $\beta\in \Delta_\mathcal{F}$. Then $\{X'_\beta\}_{\beta\in \Delta_\mathcal{F}}$ forms a $\Fr$-stable pinning of $\mathbb{G}=G_{\mathcal{F},0}/G_{\mathcal{F},0+}$.

    Take any $\omega\in \Omega_{G,\mathcal{F}}^\Fr$, and construct a lift $t'\widetilde{\omega}\in N_G(T)(F)_\mathcal{F}$ as in the proof of Proposition~\ref{prop:preserv_for_standard}.

    For $\beta\in\Delta_\mathcal{F}\setminus\{-\theta\}$ such that $\overline{\omega}\cdot \beta\neq -\theta$, we have $\widetilde{\omega}\cdot X_\beta=X_{\overline{\omega}\cdot\beta}$. Also, if $\overline{\omega}\cdot\beta=-\theta$, we also find a sequence $\beta=\beta_0,\beta_1,\dots,\beta_{2m}=\alpha$ as above. Then $n(y'_\beta)\cdot X_\beta=n(y'_\alpha)\cdot X_\alpha$ and so $\widetilde{\omega}\cdot X_\beta=X_{-\theta}$. Finally, if $\overline{\omega}\cdot (-\theta)=\beta$, we can write
    \[
        \overline{\omega}=y_\beta^{-1}vw_0,
    \]
    where $v\in \Stab_{W(T,G)}(\theta)$ and $\ell(\overline{\omega})=\ell(w_0)-\ell(y_\beta)-\ell(v)$. Then $w_0=v^{-1}y_\beta\overline{\omega}$ and 
    \[
        n(w_0)=n(v^{-1})n(y_\beta)n(\overline{\omega}).
    \]
    Thus
    \begin{align}
        \widetilde{\omega}\cdot X_{-\theta}&=n(\overline{\omega})n(y'_\alpha)\cdot X_\alpha\\
        &=n(y_\beta)^{-1}n(v^{-1})^{-1}n(y_\alpha)\cdot X_\alpha\quad\text{(by Lemma~\ref{lem:y_vs_y'})}\\
        &=n(y_\beta)^{-1}n(y_\alpha)\cdot X_\alpha.
    \end{align}
    The existence of a sequence $\beta=\beta_0,\beta_1,\dots,\beta_{2m}=\alpha$ as above implies that $n(y_\beta)^{-1}n(y_\alpha)\cdot X_\alpha=X_\beta$. Therefore we have $\widetilde{\omega}\cdot X_\beta=X_{\overline{\omega}\cdot\beta}$ for all $\beta\in \Delta_\mathcal{F}$. Then $t'\widetilde{\omega}\cdot X'_\beta=\varpi^n X'_{\overline{\omega}\cdot\beta}$ for some $n\in\Z$, but $t'\widetilde{\omega}\in G_\mathcal{F}$ implies $n=0$. Therefore $t'\widetilde{\omega}$ is a lift of $\omega$ which stabilizes a pinning $\{X'_\beta\}.$

    Now we consider the case where $G$ is of type $D_{2n+1}$.
    
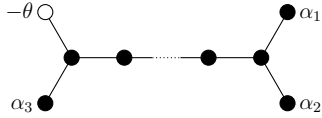
\[
        \begin{tikzpicture}[scale=2]
            \dynkin{D}[1]{}
            \dynkinLabelRoot{0}{-\theta}
            \dynkinLabelRoot{1}{\alpha_3}
            \dynkinLabelRoot{6}{\alpha_1}
            \dynkinLabelRoot{7}{\alpha_2}
        \end{tikzpicture}
    \]
    If $\Omega_{G,\mathcal{F}}^\Fr\cong \Z/2\Z$, we set $\alpha=\alpha_3$, then the above discussion can be applied in this case. If $\Omega_{G,\mathcal{F}}^\Fr\cong \Z/4\Z$, we can take a generator $\omega\in \Omega_{G,\mathcal{F}}^\Fr$ such that $\omega\cdot\alpha_1=-\theta,\ \omega\cdot(-\theta)=\alpha_2$. Then we set $\alpha=\alpha_1$ and the above discussion can be applied for this $\omega$. Thus we obtain a lift $t'\widetilde{\omega}$ and $(t'\widetilde{\omega})^n$ gives a lift of $\omega^n$ which preserves the pinning.
\end{proof}

\begin{corollary}\label{cor:preserv_pinning_twist}
    If $G$ is split and not of type $D_{2n}$, then for any inner form $G'$ of $G$, $\Fr$-stable facet $\mathcal{F}\subset\mathcal{B}(G'_{F^\ur})$ and $\omega\in \Omega_{G,\mathcal{F}}$, we can take a lift $g'_\omega\in G'(F)_\mathcal{F}$ of $\omega$ which preserves a $\Fr$-stable pinning of $\mathbb{G}'=G'_{\mathcal{F},0}/G'_{\mathcal{F},0+}$.
\end{corollary}

\begin{proof}
    We may assume that $G_{F^\ur}=G'_{F^\ur}$ and $\mathcal{F}\subset \mathcal{B}(G'_{F^\ur})=\mathcal{B}(G_{F^\ur})$ is a face of the chamber $\mathcal{C}$ chosen in this section. We may even assume that $\mathcal{C}$ is $\Fr$-stable as a chamber of $\mathcal{B}(G'_{F^\ur})$. Then the action of $\Fr$ on $G'(F^\ur)$ is expressed as $\Fr\circ\Ad(g)$ under the identification $G(F^\ur)=G'(F^\ur)$, where $g\in N_G(T)(F^\ur)$ is a lift of some $\omega\in \Omega_{G,\mathcal{F}}$. Since $G$ is not of type $D_{2n}$, $\Omega_{G,\mathcal{F}}\subset\Omega_G$ is cyclic, so we can take a generator $\omega_0$. Then Theorem~\ref{thm:fund_preserv_pinning} says that there exists a $\Fr$-stable pinning $\{X_\alpha\}$ of $\mathbb{G}=G_{\mathcal{F},x}/G_{\mathcal{F},0+}$ and a lift $g_0\in N_G(T)(F)$ of $\omega_0$ which preserves it. If $\omega=\omega_0^m$, then $g^{-1}g_0^m\in T(F^\ur)_0$. Since $H^1(\Gamma_{F^\ur/F},T(F^\ur)_0)$ is trivial, we may suppose $g=g_0^m$. Then $g_0\in G'(F)_\mathcal{F}$ and $\{X_\alpha\}$ is $\Fr$-stable as a pinning of $\mathbb{G}'$.
\end{proof}

\section{A bijective correspondence for some simple adjoint groups}\label{sect:corr_for_simple}

In this section, we verify the condition (B) in Proposition~\ref{prop:jord_bij} and give a bijective $\LLC$ for some simple adjoint groups. The result is:

\begin{theorem}\label{thm:LLC_simple_adjoint}
    The map $\LLC$ defined in Theorem~\ref{thm:general_LLC} is bijective in the following cases:
    \begin{enumerate}[label={$(\arabic*)$}]
        \item $G'$ is of type $A_n,E_6,E_8,F_4$ or $G_2$ (including all inner or outer forms).
        \item $G'$ is an inner form of ${}^3D_4.$
        \item $G'=G$ is split of type $C_n$ or $E_7$.
    \end{enumerate}
    Moreover, if $p>2$, then $\LLC$ is bijective when:
    \begin{enumerate}[label={$(\arabic*)$}]
        \setcounter{enumi}{3}
        \item $G'$ is of type $B_n$.
        \item $G'=G$ is quasi-split of type ${}^2 D_{2n}$ or $D_{2n+1}$.
    \end{enumerate}
\end{theorem}

Let $x\in \mathcal{B}(G')$ be a vertex. Then there exist a chamber $\mathcal{C}\subset\mathcal{B}(G'_{F^\ur})$ and a face $\mathcal{F}$ of $\mathcal{C}$ both of which are $\Fr$-stable and $\mathcal{F}^\Fr=\{x\}$. Now $\mathcal{C}$ specifies a basis of the affine root system $\Delta_\aff\subset \Psi_{T,G'}$ for an unramified maximally split maximal torus $T\subset G'$ on which $\Omega_{G}\rtimes\langle \Fr\rangle$ acts, and $\mathcal{F}$ defines a maximal $\Fr$-stable subset $\Delta_\mathcal{F}\subset \Delta_\aff$, which can be identified with a subset of $\widetilde{\Delta}=\dot{\Delta}_\aff$ by derivation. Then $G'(F)_x=G'(F^\ur)_{\mathcal{F}}^\Fr$ and $G'(F)_{x,r}=G'(F^\ur)_{\mathcal{F},r}^\Fr$ for $r\geq 0$. In particular, the (absolute) root system $\Phi_{\mathbb{G},\mathbb{T}}$ of $\mathbb{G}=G'_{x,0}/G'_{x,0+}$ contains $\Delta_\mathcal{F}$ as a basis. Moreover, if we put $\widetilde{\mathbb{G}}=G'_x/G'_{x,0+}$, then $\widetilde{\mathbb{G}}/\mathbb{G}\cong \Omega_{G,\mathcal{F}}$.

Now we will verify (B) for various cases. First, we consider the case where $G'$ is of type $A_n$. In that case, the highest root $\theta\in \Phi_{T,G'}$ with respect to a basis $\Delta\subset \widetilde{\Delta}$ of $\Phi_{T,G'}$ has the form
\[
    \theta=\sum_{\alpha\in\Delta}\alpha.
\]
Also, since $G'$ is adjoint, $\Delta$ forms a basis of $X^\ast(T)$. Thus for any subset $\Delta'\subset\widetilde{\Delta}$, $X^\ast(T)/\langle \Delta'\rangle$ is torsion-free. In particular, $X^\ast(Z(\mathbb{G}))=X^\ast(\mathbb{T})/\langle \Delta_\mathcal{F}\rangle$ is torsion-free, so $\mathbb{G}$ has a connected center. In this case, $\mathbb{G}(k)\to \mathbb{G}_\mathrm{ad}(k)$ is surjective and $\mathbb{G}_\mathrm{ad}(k)\cdot \rho_0=\{\rho_0\}$ for any $\rho_0\in \Irr(\mathbb{G}(k))_\cusp$. Take $s\in \widehat{\mathbb{G}}(k)$ such that $\rho_0\in \Irr(\mathbb{G}(k))_s$ and define $\widetilde{\mathbb{G}}(k)_s$ as in the proof of Proposition~\ref{prop:jord_surj}. Then (B) holds if and only if $\rho_0$ can be extended to a representation of $\widetilde{\mathbb{G}}(k)_s$. Recall that the adjoint action of $\widetilde{\mathbb{G}}(k)_s$ stabilizes $\mathbb{G}_\mathrm{ad}(k)\cdot \rho_0=\{\rho_0\}$. Also, $G'$ is of type $A_n$, so the fundamental group $\Omega_G$ is cyclic. Then $\widetilde{\mathbb{G}}(k)_s/\mathbb{G}(k)\subset \Omega_{G,\mathcal{F}}^\Fr$ is also cyclic and we can indeed extend $\rho_0$ to a representation of $\widetilde{\mathbb{G}}(k)_s$. Thus (B) holds in this case.

If $G'$ is of type ${}^3D_4,{}^2E_6,E_8,F_4$ or $G_2$, it is much easier to show that (B) holds; in these cases, $\Omega_G^\Fr$ is trivial and $\widetilde{\mathbb{G}}(k)=\mathbb{G}(k)$, so there is nothing to prove. For the remainder cases, we will check individually that the condition (B) holds.

\subsection*{When $G'=G$ is split of type $C_n$.}

The extended Dynkin diagram of $C_n$ is as below:
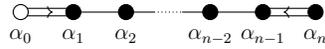
\[
    \begin{tikzpicture}[scale=2]
        \dynkin{C}[1]{}
        \dynkinLabelRoot{0}{\alpha_0}
        \dynkinLabelRoot{1}{\alpha_1}
        \dynkinLabelRoot{2}{\alpha_2}
        \dynkinLabelRoot{3}{\alpha_{n-2}}
        \dynkinLabelRoot{4}{\alpha_{n-1}}
        \dynkinLabelRoot{5}{\alpha_n}
    \end{tikzpicture}
\]
Also, $\Omega_G\cong \Z/2\Z$ and its unique nontrivial element $\omega$ acts on the diagram as $\alpha_i\mapsto \alpha_{n-i}$.

Since $G'=G$ is split, the maximal $\Fr$-stable facet $\mathcal{F}$ is a vertex $x$ and $\#(\widetilde{\Delta}\setminus\Delta_x)=1$. Then, $\Omega_{G,x}$ is trivial unless $n=2m$ and $\alpha_m\in \widetilde{\Delta}\setminus \Delta_x$ is the middle node. Thus we may suppose so. Then we have
\[
    \mathbb{G}=(\Sp_{2m}\times \Sp_{2m})/\Delta(Z),
\]
where $Z\subset \Sp_{2m}$ is the center and $\Delta\colon \Sp_{2m}\to \Sp_{2m}\times \Sp_{2m}$ is the diagonal morphism. Take $s\in \widehat{\mathbb{G}}(k)$ and $\rho_0\in \Irr(\mathbb{G}(k))_{s,\cusp}$. Since $\widetilde{\mathbb{G}}(k)_s/\mathbb{G}(k)\subset \Omega_G$ is cyclic, it suffices to show that $\widetilde{\mathbb{G}}(k)_s$ stabilizes $\rho_0$. Then we only have to consider the case where $\widetilde{\mathbb{G}}(k)_s/\mathbb{G}(k)\cong \Z/2\Z$ and $\mathbb{G}_\mathrm{ad}(k)\cdot \rho_0$ is not a singleton. Consider a morphism
\[
    p\colon \Sp_{2m}\times\Sp_{2m}\to \mathbb{G}.
\]
Restricting $\rho_0$ along $p$, we obtain a representation of $\Sp_{2m}(k)\times\Sp_{2m}(k)$. Choose an irreducible component $\rho_0'=\rho_1'\boxtimes\rho_2'$. The adjoint action of $\mathrm{PSp}_{2m}(k)$ on $\Irr(\Sp_{2m}(k))$ factors through $\Z/2\Z$, and we write $\tau$ for the nontrivial counterpart. Since $\omega\in \widetilde{\mathbb{G}}(k)_s/\mathbb{G}(k)$ fixes the orbit $\mathbb{G}_{\mathrm{ad}}(k)\cdot\rho_0=(\mathrm{PSp}_{2m}\times\mathrm{PSp}_{2m})(k)\cdot\rho_0$, the $(\mathrm{PSp}_{2m}\times\mathrm{PSp}_{2m})(k)$-orbit of $\rho'_0$ must be fixed by $\omega$. Theorem~\ref{thm:fund_preserv_pinning} says that some lift of $\omega$ preserves a pinning of $\mathbb{G}$, so $\omega$ acts on $\Irr(\Sp_{2m}(k)\times\Sp_{2m}(k))$ as $\rho_1\boxtimes\rho_2\mapsto \rho_2\boxtimes\rho_1$. Thus we have $\rho'_2=\rho'_1$ or $\tau\cdot \rho'_1$.

If $\rho'_1=\tau\cdot\rho'_1$, then $\rho'_1$ can be extended to a representation of $\mathrm{PSp}_{2m}(k)$. Thus $\rho'_0$ can be extended to a representation $\widetilde{\rho}_0$ of $\mathbb{G}_\mathrm{ad}(k)$, and $\rho_0$ is obtained as the restriction of $\widetilde{\rho}_0$. Then $\mathbb{G}_\mathrm{ad}\cdot\rho_0=\{\rho_0\}$, which contradicts the assumption. Therefore we have $\rho'_1\neq\tau\cdot\rho'_1$. Since the nontrivial adjoint action of $\mathbb{G}(k)$ on $\Irr(\Sp_{2m}(k)\times\Sp_{2m}(k))$ is written as $\rho'_1\boxtimes\rho'_2\mapsto \tau\cdot\rho'_1\boxtimes\tau\cdot\rho'_2$, it does  not fix $\rho'_0=\rho'_1\boxtimes\rho'_1$ or $\rho'_1\boxtimes \tau\cdot\rho'_1$. Thus 
\[
    \rho_0=\operatorname{Ind}(\rho'_0)=\operatorname{Ind}(\tau\cdot \rho'_0),
\]
which is fixed by the action of $\omega$. Therefore (B) holds in this case.

\subsection*{When $G'$ is of type ${}^1E_6$ or $G'=G$ is split of type $E_7$.}

In these cases, the root system $\Phi_{S,G}$ of $G$ is simply laced. Then, the root system of $H$ defined as in Section~\ref{sec:const_of_H} is additively closed in $\Phi_{S,G}$. Thus we can construct $H$ as a subgroup of $G$. According to \cite[Proposition 14.6.1, 14.8.4]{Kaletha_Prasad_2023}, there exists an embedding of the building $\iota\colon \widetilde{\mathcal{B}}(H_{F^\ur})\hookrightarrow\widetilde{\mathcal{B}}(G_{F^\ur})$ which is equivariant under the action of $H(F^\ur)$ and, for any split maximal torus $T\subset H_{F^\ur}$, the restriction of $\iota$ provides an isomorphism of the affine spaces
\[
    \widetilde{\mathcal{A}}(T,H_{F^\ur})\to \mathcal{A}(T,G_{F^\ur})
\]
over $V(T)=X_\ast(T)\otimes \R$. Moreover, we can choose such an embedding that is stable under $\Fr$ (\cite[Corollary 14.7.3]{Kaletha_Prasad_2023}). However, the equality of the rank of the split centers $Z(H)_\mathrm{s}$ and $Z(G)_\mathrm{s}$ implies that such an embedding is unique; indeed, for an unramified elliptic maximal torus $T\subset H$, it is also elliptic in $G$, so $\widetilde{\mathcal{A}}(T_{F^\ur},H_{F^\ur})$ and $\mathcal{A}(T_{F^\ur},G_{F^\ur})$ contain only one $\Fr$-fixed point, respectively. Then \cite[Lemma 14.9.1]{Kaletha_Prasad_2023} shows the uniqueness. In particular, restricting $\iota$ to the apartment of a quasi-split torus $S$ of $H$, we obtain an embedding $\iota_0$ of the apartment as in Proposition~\ref{prop:embd_apart}. Also, any $\xi\in H^1(\Gamma_F,H)$ provides an embedding $H_\xi\hookrightarrow G_{\overline{\xi}}$.

Now the extended Dynkin diagram of $E_6$ is as follows:
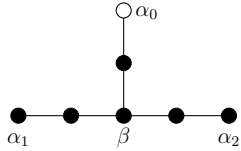
\[
    \begin{tikzpicture}[scale=2]
        \dynkin{E}[1]{6}
        \dynkinLabelRoot{0}{\alpha_0}
        \dynkinLabelRoot{1}{\alpha_1}
        \dynkinLabelRoot{4}{\beta}
        \dynkinLabelRoot{6}{\alpha_2}
    \end{tikzpicture}
\]
Here, the simple roots $\alpha_1$ and $\alpha_2$ appear in the highest root $\theta=-\alpha_0$ with multiplicity one. Thus, if $\{\alpha_0,\alpha_1,\alpha_2\}\not\subset\Delta_\mathcal{F}$, $X^\ast(S')/\langle \Delta_\mathcal{F}\rangle$ is torsion free and $\mathbb{G}$ has a connected center. Also, the fundamental group $\Omega_G\cong \Z/3\Z$ acts on the diagram as rotations around $\beta$. Then, if the  center of $\mathbb{G}$ is disconnected and $\Omega_{G,\mathcal{F}}\cong \Z/3\Z$, the possible choices of $\Delta_\mathcal{F}$ are:
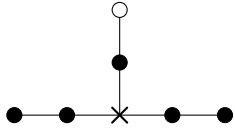
\[
    \begin{tikzpicture}[scale=2]
        \dynkin{E}[1]{***x**}
    \end{tikzpicture}
\]
for the split form $G'=G$, and
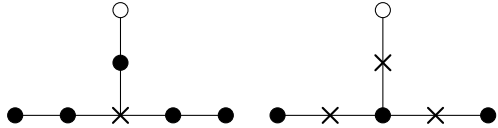
\[
    \begin{tikzpicture}[scale=2]
        \dynkin{E}[1]{***x**}
    \end{tikzpicture}\quad
    \begin{tikzpicture}[scale=2]
        \dynkin{E}[1]{*xx*x*}
    \end{tikzpicture}
\]
for the non-split inner form $G'$, where the removed vertices are expressed as $\times$.

Similarly, we consider the case where $G'=G$ is split of type $E_7$. The extended Dynkin diagram is below:
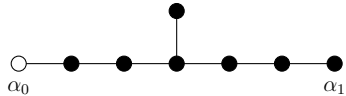
\[
    \begin{tikzpicture}[scale=2]
        \dynkin{E}[1]{7}
        \dynkinLabelRoot{0}{\alpha_0}
        \dynkinLabelRoot{7}{\alpha_1}
    \end{tikzpicture}
\]
Similarly to the case of $E_6$, we may assume $\{\alpha_0,\alpha_1\}\subset \Delta_\mathcal{F}$, otherwise the center of $\mathbb{G}$ is connected. Also, the nontrivial element of the fundamental group $\Omega_G\cong \Z/2\Z$ acts on the diagram as a reflection along the middle vertical edge. Then the possible choices of $\Delta_\mathcal{F}$ when $\Omega_{G,\mathcal{F}}\cong \Z/2\Z$ are as follows.
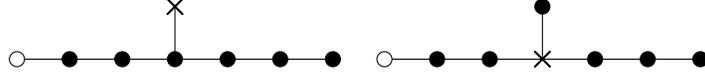
\[
    \begin{tikzpicture}[scale=2]
        \dynkin{E}[1]{*x*****}
    \end{tikzpicture}\quad
    \begin{tikzpicture}[scale=2]
        \dynkin{E}[1]{***x***}
    \end{tikzpicture}
\]

In all of those cases, $\mathbb{G}$ is a product of reductive groups of type ${}^1 A_n$. Now, any cuspidal representation of a group of type ${}^1A_n$ comes from a \emph{non-singular} character $\theta$ in the sense of \cite[Definition 5.15]{DL}. Then applying \cite[Theorem 2.7.7]{kaletha2021supercuspidallpackets} to $G(k)=\widetilde{\mathbb{G}}(k)_s$ and $S(k)=S^\circ(k)=\mathbb{S}'(k)$, we obtain a bijection
\[
    \Irr(\widetilde{\mathbb{G}}(k)_s)_s\to \Irr(N_{\widetilde{\mathbb{G}}(k)_s}(\mathbb{S}'),\theta)=\{\rho\in \Irr(N_{\widetilde{\mathbb{G}}(k)_s}(\mathbb{S}'))\mid \rho|_{\mathbb{S}'(k)}\ \text{is $\theta$-isotypic}\}.
\]
As in the proof of Proposition~\ref{prop:jord_surj}, we have
\[
    N_{\widetilde{\mathbb{G}}(k)_s}(\mathbb{S}')/\mathbb{S}'(k)=N_{\widetilde{\mathbb{H}}(k)}(\mathbb{S}')/\mathbb{S}'(k).
\]
However, we have obtained $H_\xi$ as a subgroup of $G_{\overline{\xi}}=G'$, we also get an equality
\[
    N_{\widetilde{\mathbb{G}}(k)_s}(\mathbb{S}')=N_{\widetilde{\mathbb{H}}(k)}(\mathbb{S}').
\]
Moreover, $\mathbb{H}=\mathbb{S}'$ by the non-singularity of $\theta$, so $N_{\widetilde{\mathbb{H}}(k)}(\mathbb{S}')=\widetilde{\mathbb{H}}(k)$. Since $\theta$ can be extended to a representation of $\widetilde{\mathbb{H}}(k)$ as in the proof of Corollary~\ref{cor:tame_twist}, we have
\[
    \#\Irr(\widetilde{\mathbb{H}}(k),\theta)=\frac{\abs{\widetilde{\mathbb{H}}(k)}}{\abs{\mathbb{H}(k)}}=\frac{\abs{\widetilde{\mathbb{H}}(k)}}{\abs{\mathbb{H}^1(k)}}\cdot \frac{\abs{\mathbb{H}^1(k)}}{\abs{\mathbb{H}(k)}}.
\]
The isomorphism \eqref{eq:H_and_G_s} shows that $\abs{\widetilde{\mathbb{H}}(k)}/\abs{\mathbb{H}^1(k)}=\abs{\widetilde{\mathbb{G}}(k)_s}/\abs{\mathbb{G}(k)}$. 
Also, \cite[Lemma 15.7]{FOS} implies that $\abs{\mathbb{H}^1(k)}/\abs{\mathbb{H}(k)}=\#\Irr(\mathbb{H}^1(k))_{\unip,\cusp}=\#\Irr(\mathbb{H}^1(k))_\unip$. Then, by Corollaries~\ref{cor:regular_cuspidal}~and~\ref{cor:jord_G_H1}, we have
\[
    \frac{\abs{\mathbb{H}^1(k)}}{\abs{\mathbb{H}(k)}}=\#\Irr(\mathbb{H}^1(k))_\unip=\#\Irr(\mathbb{G}(k))_s.
\]
Therefore we have
\[
    \#\Irr(\widetilde{\mathbb{G}}(k)_s)_s=\#\Irr(\widetilde{\mathbb{H}}(k),\theta)=\frac{\abs{\widetilde{\mathbb{G}}(k)_s}}{\abs{\mathbb{G}(k)}}\cdot \#\Irr(\mathbb{G}(k))_s.
\]
Then any $\rho\in \Irr(\mathbb{G}(k))_s$ can be extended to a representation of $\widetilde{\mathbb{G}}(k)_s$, so (B) holds.

\subsection*{When $G'$ is of type $B_n$ or quasi-split of type ${}^2 D_{2n}$ or $D_{2n+1}$.}

According to \cite[Proposition 7.10 (ii)]{Lust_Stevens_2020} and its proof, we have the following:

\begin{proposition}\label{prop:outer_stop_SO}
    Let $\mathbb{G}=\SO^\pm_{2n},\ n\geq 2$ be the quasi-split special orthogonal group over a finite field $k$ with odd characteristic and $\tau\colon \mathbb{G}\to\mathbb{G}$ be the pinning-preserving outer automorphism of order two. Suppose that an irreducible cuspidal representation $\sigma\in \Irr(\mathbb{G}(k))$ belongs to the Lusztig series $\Irr(\mathbb{G}(k))_s$ associated with $s\in \widehat{\mathbb{G}}(k)=\SO_{2n}^{\pm}(k)\subset\GL_{2n}(k)$. Then the following are equivalent:
    \begin{enumerate}[label=(\alph*)]
        \item The Lusztig series $\Irr(\mathbb{G}(k))_s$ is stabilized by $\tau$.
        \item At least 1 or $-1$ is an eigenvalue of $s$.
        \item The representation $\sigma$ is stabilized by $\tau$.
    \end{enumerate}
\end{proposition}

We use this fact to verify the condition (B). First, we consider the case where $G'$ is of type $B_n$. The extended Dynkin diagram is as follows:
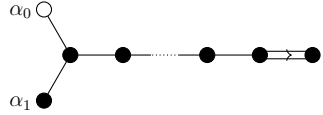
\[
    \begin{tikzpicture}[scale=2]
        \dynkin{B}[1]{}
        \dynkinLabelRoot{0}{\alpha_0}
        \dynkinLabelRoot{1}{\alpha_1}
    \end{tikzpicture}
\]
Also, $\Omega_{G}\cong \Z/2\Z$ acts on the diagram as swapping $\alpha_0$ and $\alpha_1$ with each other. Since the multiplicity of $\alpha_1$ in the highest root $\theta=-\alpha_0$ is one, the center of $\mathbb{G}=G'_{\mathcal{F},0}/G'_{\mathcal{F},0+}$ is connected if $\{\alpha_0,\alpha_1\}\not\subset\Delta_\mathcal{F}$. When $\{\alpha_0,\alpha_1\}\subset\Delta_\mathcal{F}$, we have
\[
    \mathbb{G}\cong \SO^{\pm}_{2l}\times \SO_{2m+1},\quad l+m=n,\quad l\geq 2,\ m\geq 0.
\]
By Theorem~\ref{thm:fund_preserv_pinning} and Corollary~\ref{cor:preserv_pinning_twist}, we can choose a lift of the nontrivial element $\omega\in \Omega_{G,\mathcal{F}}$ which preserves a pinning of $\mathbb{G}$, so it acts on $\Irr(\SO^\pm_{2l}(k)\times \SO_{2m+1}(k))$ as $\rho_1\boxtimes\rho_2\mapsto (\tau\cdot\rho_1)\boxtimes\rho_2$. Then Proposition~\ref{prop:outer_stop_SO} implies the condition (B) for any $\rho\in\Irr(\mathbb{G}(k))_\cusp$.

When $G'=G$ is quasi-split of type $D_n$, the extended Dynkin diagram is as follows:
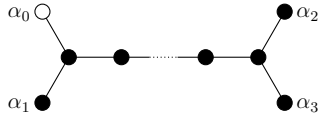
\[
    \begin{tikzpicture}[scale=2]
        \dynkin{D}[1]{}
        \dynkinLabelRoot{0}{\alpha_0}
        \dynkinLabelRoot{1}{\alpha_1}
        \dynkinLabelRoot{6}{\alpha_2}
        \dynkinLabelRoot{7}{\alpha_3}
    \end{tikzpicture}
\]
Since the simple roots $\alpha_1,\alpha_2,\alpha_3$ appear in the highest root $\theta=-\alpha_0$ with multiplicity one, we may assume that $\{\alpha_0,\dots,\alpha_3\}\subset \Delta_x$, otherwise the center of $\mathbb{G}=G_{x,0}/G_{x,0+}$ is connected. Then $\mathbb{G}$ is isomorphic to
\[
    (\SO^\pm_{2l}\times \SO_{2m})/\Delta(Z)
\]
where $l,m\geq 2,\ l+m=n,\ Z=Z(\SO^\pm_{2l})=Z(\SO_{2m})\cong \Z/2\Z$ and $\Delta\colon Z\hookrightarrow \SO^\pm_{2l}\times\SO_{2m}$ is the diagonal morphism.
We put
\[
    \mathbb{G}'=\SO^\pm_{2l}\times \SO_{2m}
\]
respectively, then we have a canonical map $p\colon \mathbb{G}'\to \mathbb{G}$.

Take any $\rho_0\in \Irr(\mathbb{G}(k))_s,\ s\in \widehat{\mathbb{G}}(k)$. We want to show that $\rho_0$ can be extended to a representation of $\widetilde{\mathbb{G}}(k)_s$. We may assume that $\widetilde{\mathbb{G}}(k)_s\neq \mathbb{G}(k)$. If $G$ is not split of type $D_{2n'}$, we have $\widetilde{\mathbb{G}}(k)_s/\mathbb{G}(k)\subset \Omega_{G,\mathcal{F}}\cong \Z/2\Z$, so $\widetilde{\mathbb{G}}(k)_s/\mathbb{G}(k)\cong \Z/2\Z$.
Then it suffices to show that $\widetilde{\mathbb{G}}(k)_s$ stabilizes $\rho_0$. Since $\widetilde{\mathbb{G}}(k)_s$ stabilizes $\mathbb{G}_\mathrm{ad}(k)\cdot \rho_0$, we may also assume that $\mathbb{G}_\mathrm{ad}(k)\cdot \rho_0\neq \{\rho_0\}$. Restrict $\rho_0$ along $p$ and take an irreducible component $\rho'_0=\rho'_1\boxtimes\rho'_2$ of it. We denote by $\sigma$ the nontrivial adjoint action on $\SO^\pm_{2m}$  or $\SO_{2l}$. Then, as in the case of $G=\mathrm{PSp}$, we have $\sigma\cdot\rho'_1\neq \rho'_1$ or $\sigma\cdot\rho'_2\neq \rho'_2$, so $\sigma\cdot \rho'_0\neq \rho'_0$ and $\rho_0=\operatorname{Ind}(\rho'_0)$. The nontrivial element $\omega\in \widetilde{\mathbb{G}}(k)_s/\mathbb{G}(k)$ acts on $\Irr(\mathbb{G}'(k))$ as $\rho'_1\boxtimes\rho'_2\mapsto (\tau\cdot \rho'_1)\boxtimes (\tau\cdot\rho'_2)$, where $\tau$ is the outer automorphism as in Proposition~\ref{prop:outer_stop_SO}. Then $\omega$ fixes $\rho'_0$ and so $\rho_0$. Thus $\rho_0$ can be extended to a representation of $\widetilde{\mathbb{G}}(k)_s$ and (B) holds in this case. 

\begin{example}
    Suppose that $G=\SO_5$ is the special orthogonal group of degree five. Then $\widehat{G}=\Sp_4(\C)$. Let 
    \[
        \varphi_1=\varphi_{\theta_0}\colon W_F\times\SL_2(\C)\twoheadrightarrow W_F\to \Ldual{T'_1}\hookrightarrow\SL_2(\C)\times\langle \Fr\rangle
    \]
    be the depth-zero supercuspidal L-parameter corresponding a regular depth-zero character $\theta_0\colon T'_1(F)\to \C^\times$ of an elliptic maximal torus $T'_1\subset \PGL_2$ as in \cite[p.~828]{DR}. Also, Let
    \[
        \varphi_2=\mathrm{Std}\colon W_F\times\SL_2(\C)\twoheadrightarrow\SL_2(\C)\hookrightarrow\SL_2(\C)\times\langle \Fr\rangle
    \]
    be the L-parameter coming from the standard representation of $\SL_2(\C)$. Then 
    \[
        Z_{\SL_2(\C)}(\varphi_i)=Z(\SL_2(\C))
    \]
    for $i=1,2$. Let $\epsilon_1=\epsilon_2$ be the nontrivial character of $Z(\SL_2(\C))\cong \Z/2\Z$. Now, $(\varphi,\epsilon)=(\varphi_1\boxtimes\varphi_2,\epsilon_1\boxtimes\epsilon_2)$ is an enhanced cuspidal L-parameter for $G$, where we see $\varphi$ as an L-parameter of $\Ldual{G}=\Sp_4(\C)\times\langle \Fr\rangle$ by
    \[
        \varphi_1\boxtimes\varphi_2\colon W_F\times\SL_2(\C)\to \SL_2(\C)\times\SL_2(\C)\times\langle \Fr\rangle\subset \Sp_4(\C)\times \langle \Fr\rangle.
    \]
    Then $Z_{\widehat{G}}(\varphi)=Z_{\SL_2(\C)}(\varphi_1)\times Z_{\SL_2(\C)}(\varphi_2)$. In this case, we have
    \[
        \Ldual{H}=\Ldual{T'_1}\times\SL_2(\C)\hookrightarrow \Sp_4(\C)\times\langle \Fr\rangle.
    \]
    Take a maximal torus $\widehat{T}_2\subset \SL_2(\C)$ and put $\Ldual{S}=\Ldual{T'_1}\times \widehat{T}_2$. Then $S=T'_1\times T_2\subset H=T'_1\times \PGL_2$ is a quasi-split maximal torus. Also, we have a character $\theta\colon S(F)_0\to \C^\times$, which is expressed as 
    \[
        S(F)_0=T'_1(F)\times T'_2(F)_0\twoheadrightarrow T'_1(F)\xrightarrow{\theta_0}\C^\times.
    \] 
    For convenience, we take a basis $\alpha,\beta$ of the root system $\Phi_{S,G}$ such that $\beta$ is a long root and $X^\ast(T'_1)=\Z\cdot \alpha,\ X^\ast(T_2)=\Z\cdot (-\alpha-\beta)$. The embedding $\iota_0$ of $\widetilde{\mathcal{A}}(S,H)$ into $\mathcal{A}(S,G)$ defined in Proposition~\ref{prop:embd_apart} is the following:
    
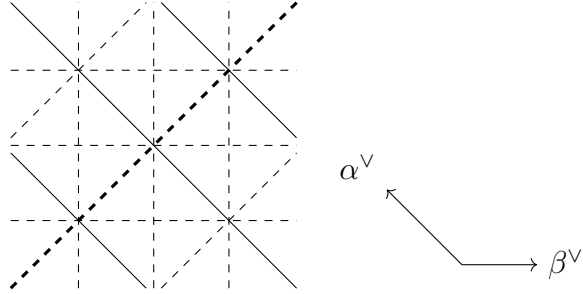
\[
        \begin{tikzpicture}
            \draw[dashed] (-1.9,0)--(1.9,0);
            \draw[dashed] (0,-1.9)--(0,1.9);
            \draw[dashed] (-1.9,1)--(1.9,1);
            \draw[dashed] (1,-1.9)--(1,1.9);
            \draw[dashed] (-1,-1.9)--(-1,1.9);
            \draw[dashed] (-1.9,-1)--(1.9,-1);
            \draw[very thick,dashed] (-1.9,-1.9)--(1.9,1.9);
            \draw (-1.9,1.9)--(1.9,-1.9);
            \draw[dashed] (0.1,-1.9)--(1.9,-0.1);
            \draw[dashed] (-1.9,0.1)--(-0.1,1.9);
            \draw (0.1,1.9)--(1.9,0.1);
            \draw (-1.9,-0.1)--(-0.1,-1.9);
        \end{tikzpicture}\quad
        \begin{tikzpicture}
            \draw[->] (0,0)--(1,0) node[right]{$\beta^\vee$};
            \draw[->] (0,0)--(-1,1) node[above left]{$\alpha^\vee$};
        \end{tikzpicture}
    \]
    Here, the hyperplane corresponding to an affine root $\psi\in\Psi_{S,G}\setminus \Psi_{S,H}$ is expressed by a dashed line. Also, $\Fr$ acts on this apartment by reflection along the dashed thick line in the above picture.

    Now $Z(\widehat{H})=Z(\SL_2(\C))\times Z(\SL_2(\C))\subset \widehat{S}_0\times \SL_2(\C)$, and then $\epsilon$ defines a pure inner twist $\xi$ of $H$. We can calculate that $H_\xi=T'_1\times \mathrm{P}D^\times$, where $D$ is the quaternion division algebra over $F$. Also, we have an unramified maximally split maximal torus $S'=T'_1\times T'_2,\ T'_2\subset \mathrm{P}D^\times$ and the embedding $\widetilde{\mathcal{A}}(S',H_\xi)\hookrightarrow \mathcal{A}(S',G)$ is as follows:
    
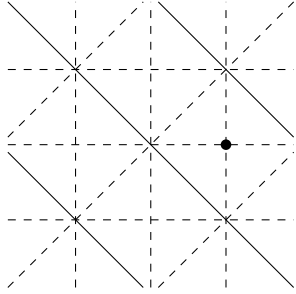
\[
        \begin{tikzpicture}
            \draw[dashed] (-1.9,0)--(1.9,0);
            \draw[dashed] (0,-1.9)--(0,1.9);
            \draw[dashed] (-1.9,1)--(1.9,1);
            \draw[dashed] (1,-1.9)--(1,1.9);
            \draw[dashed] (-1,-1.9)--(-1,1.9);
            \draw[dashed] (-1.9,-1)--(1.9,-1);
            \draw[dashed] (-1.9,-1.9)--(1.9,1.9);
            \draw (-1.9,1.9)--(1.9,-1.9);
            \draw[dashed] (0.1,-1.9)--(1.9,-0.1);
            \draw[dashed] (-1.9,0.1)--(-0.1,1.9);
            \draw (0.1,1.9)--(1.9,0.1);
            \draw (-1.9,-0.1)--(-0.1,-1.9);
            \fill (1,0) circle (2pt);
        \end{tikzpicture}
    \]
    Here, $\Fr$ acts as the rotation of degree $\pi$ around the point marked above. This point is a vertex of $\mathcal{B}(G)$ and we denote it by $x$. Then, the representation corresponding to $(\varphi,\epsilon)$ is obtained as compact induction of an irreducible representation $\rho$ of $G(F)_x$ which contains a cuspidal representation $\rho_0\in \Irr(\mathbb{G}(k))_s$, where $\mathbb{G}=G_{x,0}/G_{x,0+}$ and $s\in \widehat{\mathbb{G}}(k)$ is the semisimple element corresponding to $\theta$. Now $\mathbb{G}=(\SL_2\times\SL_2)/\Delta(Z(\SL_2))$, where $\Delta\colon \SL_2\hookrightarrow\SL_2\times\SL_2$ is the diagonal embedding. Moreover, $\mathbb{S}'=S'_0/S'_{0+}$ can be written as 
    \[
        \mathbb{S}'=(\mathbb{S}'_1\times\mathbb{S}'_2)/\Delta(Z(\SL_2)),
    \]
    where $\mathbb{S}'_1,\mathbb{S}'_2\subset\SL_2$ are elliptic maximal tori. We can see that the image of $X_\ast(\mathbb{S}'_1),X_\ast(\mathbb{S}'_2)\to X_\ast(\mathbb{S}')$ are $\Z\cdot \beta^\vee$ and $\Z\cdot (-\alpha^\vee-\beta^\vee)$. Also, $X_\ast(\mathbb{T}'_1)=\Z\cdot \alpha^\vee$ and $X_\ast(\mathbb{T}'_2)=\Z\cdot (-\alpha^\vee-2\beta^\vee)$. Thus, under the identification $\mathbb{S}'_1=\mathbb{S}'_2=\mathbb{T}'_1=\mathbb{T}'_2=\mathbb{G}_\mathrm{m}$ (on which $\Fr$ acts as $\Fr(a)=a^{-q}$), we have
    \[
        \mathbb{S}'_1\times\mathbb{S}'_2\to \mathbb{S}'=\mathbb{T}'_1\times\mathbb{T}'_2;\quad (s_1,s_2)\mapsto (s_1^{-1}s_2^{-1},s_1^{-1}s_2),
    \]
    and the character $\theta$ on $\mathbb{S}'(k)$ factors as $\mathbb{S}'(k)\to \mathbb{T}'_1(k)\xrightarrow{\theta_0}\C^\times$. The representation $\rho_0$ is an irreducible component of the Deligne--Lusztig character $R^{\mathbb{G}}_{\mathbb{S}'}(\theta)$ defined from this character $\theta$, and $\LLC(\varphi,\epsilon)$ is equal to $\cind_{G(F)_x}^{G(F)}\rho$ for an extension $\rho$ of $\rho_0$.
\end{example}

\section{The Hiraga--Ichino--Ikeda conjecture}\label{sect:HII_conj}

Fix a character $\psi\colon F\to \C^\times$ which is trivial on $\mathcal{O}_F$, but not on $\varpi^{-1}\mathcal{O}_F$. In \cite[p.~285]{HII}, Hiraga, Ichino and Ikeda fix a Haar measure on $G/A$ with respect to $\psi$, where $A\subset G$ is the split center of $G$. Then they present a conjecture that relates the formal degree of irreducible essentially square-integrable representations with the corresponding enhanced L-parameters:

\begin{conjecture}[{\cite[Conjecture 1.4]{HII}}]\label{conj:HII}
    Let $\pi\in \Irr(G(F))$ be an essentially square-integrable representation and $(\varphi,\epsilon)$ the corresponding enhanced L-parameter (via the conjectural local Langlands correspondence). Then
    \[
        \fdeg(\pi)=\frac{\dim(\epsilon)}{\abs{\mathcal{S}^\natural_\varphi}}\cdot\abs{\gamma(0,\varphi,\Ad,\psi)},
    \]
    where $\mathcal{S}^\natural_\varphi=\pi_0(Z_{(G/A)^\wedge}(\varphi))$ and $\gamma(s,\varphi,\Ad,\psi)$ is the \emph{adjoint $\gamma$-factor} defined as in {\upshape\cite[p.~286]{HII}}.
\end{conjecture}

In \cite[Theorem 3]{FOS}, Feng, Opdam and Solleveld show that the correspondence given in Theorem~\ref{thm:LLC_for_unipotent} satisfies the above conjecture. Now we extend this result to our correspondence:

\begin{theorem}\label{thm:HIIconj}
    Suppose that the map $\LLC$ given in Theorem~\ref{thm:general_LLC} is bijective. Then Conjecture~\ref{conj:HII} holds with respect to this parametrization.
\end{theorem}

\begin{proof}
    Take any $\pi\in \Irr(G'(F))_{0,\cusp}$ and put $(\varphi,\epsilon)=\LLC^{-1}(\pi)$. Then $\varphi$ factors through $\Ldual{H}\hookrightarrow\Ldual{G}$ and defines an L-parameter $\varphi_H$ of $\Ldual{H}$. Also, $(\varphi_H,\epsilon)$ defines an irreducible supercuspidal representation of an inner twist $H'$ of $H$, which is, up to twist by characters, unipotent. Thus we have a unipotent supercuspidal representation $\pi_\unip\in \Irr(H'(F))_{\unip,\cusp}$. By the construction of $\LLC$, we have a maximal torus $S'$ of both $H'$ and $G'$ and a vertex $x\in \mathcal{A}(S',H')\hookrightarrow\mathcal{A}(S',G')$. Also, using the notation in the proof of Proposition~\ref{prop:jord_surj}, we obtain $\rho\in \Irr(\widetilde{\mathbb{G}}(k))_{s,\cusp}$ and $\rho_\unip=\widetilde{J}_0'^{-1}(\rho)\in \Irr(\widetilde{\mathbb{H}}(k))_{\unip,\cusp}$ such that
    \[
        \pi=\cind_{G'(F)_x}^{G'(F)}\rho,\quad \pi_\unip=\cind_{H'(F)_x}^{H'(F)}\rho_\unip
    \]
    for a semisimple element $s\in \widehat{\mathbb{G}}(k)$. Moreover, $\rho$ can be expressed as 
    \[
        \rho=\cind_{\widetilde{\mathbb{G}}(k)_s}^{\widetilde{\mathbb{G}}(k)}\rho_0
    \]
    for some $\rho_0\in \Irr(\widetilde{\mathbb{G}}(k)_s)_s$. Thus $\pi=\cind_{G'(F)_{x,s}}^{G'(F)}\rho_0$, where $G'(F)_{x,s}\subset G'(F)_x$ is the preimage of $\widetilde{\mathbb{G}}(k)_s$ along $G'(F)_x\twoheadrightarrow \widetilde{\mathbb{G}}(k)$. Then the formal degree of $\pi$ is written as
    \[
        \fdeg(\pi)=\frac{\dim\rho_0}{\vol(G'(F)_{x,s})}=\frac{\dim\rho_0}{[\widetilde{\mathbb{G}}(k)_s:\mathbb{G}(k)]\cdot \vol(G'(F)_{x,0})}.
    \]
    Let $\mathbb{H}^1\subset \widetilde{\mathbb{H}}$ as in Lemma~\ref{lem:isom_comp_H1} and $H'(F)_x^1\subset H'(F)_x$ be the preimage of it along $H'(F)_x\twoheadrightarrow\widetilde{\mathbb{H}}(k)$. Then we also have
    \[
        \fdeg{\pi_\unip}=\frac{\dim\rho_\unip}{\vol(H'(F)_x)}=\frac{\dim\rho_\unip}{[\widetilde{\mathbb{H}}(k):\mathbb{H}^1(k)]\cdot\vol(H'(F)_x^1)}.
    \]
    As in the proof of Proposition~\ref{prop:jord_surj}, $[\widetilde{\mathbb{G}}(k)_s:\mathbb{G}(k)]=[\widetilde{\mathbb{H}}(k):\mathbb{H}^1(k)]$, so
    \[
        \frac{\fdeg\pi}{\fdeg\pi_\unip}=\frac{\dim\rho_0}{\dim\rho_\unip}\cdot \frac{\vol(H'(F)_x^1)}{\vol(G'(F)_{x,0})}.
    \]
    Let $I^+_{G'}\subset G'(F)$ and $I^+_{H'}\subset H'(F)$ be pro-$p$ Iwahori subgroups which are contained in $G'(F)_{x,0}$ and $H'(F)_{x,0}$, respectively. Then, according to \cite[p.~295]{Gross},
    \[
        \vol(I^+_{G'})=q^{N_{G'}-\dim G'},\quad \vol(I^+_{H'})=q^{N_{H'}-\dim H'},
    \]
    where $N_{G'}$ and $N_{H'}$ are the number of positive roots of $G'$ and $H'$, respectively. Since $H'$ and $G'$ contain the maximal torus $S'$ in common, we have
    \[
        N_{G'}-N_{H'}=(\dim G'-\dim H')/2.
    \]
    Moreover, the image of $I^+_{G'}$ in $\mathbb{G}(k)$ is the unipotent radical of a Borel subgroup of $\mathbb{G}$, so it is a Sylow $p$-group of $\mathbb{G}(k)$. Also, the same is true for $\mathbb{H}^1(k)$ because $\pi_0(\mathbb{H}^1)$ is a $p'$-group by the proof of \cite[Lemma 11.2.1]{Digne_Michel_2020}.
    Thus we have
    \begin{align}
        \frac{\vol(H'(F)_x^1)}{\vol(G'(F)_{x,0})}&=\frac{\vol(I^+_{H'})}{\vol(I^+_{G'})}\cdot \frac{\abs{\mathbb{H}^1(k)}/\abs{\mathbb{H}^1(k)}_p}{\abs{\mathbb{G}(k)}/\abs{\mathbb{G}(k)}_p}\\
        &=q^{(\dim G'-\dim H')/2}\cdot \frac{\abs{\mathbb{H}^1(k)}_{p'}}{\abs{\mathbb{G}(k)}_{p'}},
    \end{align}
    where $\abs{\Gamma}_p$ (resp.\ $\abs{\Gamma}_{p'}$) is the $p$-torsion order (resp.\ $p'$-torsion order) of a finite group $\Gamma$. On the other hand, the bijectivity of $\LLC$ implies that $\rho_0|_{\mathbb{G}(k)}$ is still irreducible. Then $\rho_0|_{\mathbb{G}(k)}$ and $\rho_\unip|_{\mathbb{H}^1(k)}$ are related via Lusztig's Jordan decomposition as in Theorem~\ref{thm:Lusztig_jord_decomp}. In this case, \cite[Proposition 11.5.6]{Digne_Michel_2020} says that 
    \[
        \frac{\dim\rho_0}{\dim\rho_\unip}=\frac{\abs{\mathbb{G}(k)}_{p'}}{\abs{\mathbb{H}^1(k)}_{p'}}.
    \]
    Thus we have
    \[
        \frac{\fdeg\pi}{\fdeg\pi_\unip}=q^{(\dim G'-\dim H')/2}.
    \]

    Next, we compare the adjoint $\gamma$-factors. Recall that it is defined as
    \[
        \gamma(s,\varphi,\Ad,\psi)=\epsilon(s,\Ad\circ\varphi,\psi)\cdot \frac{L(1-s,(\Ad\circ\varphi)^\vee)}{L(s,\Ad\circ\varphi)}.
    \]
    Put $\widehat{\mathfrak{g}}=\Lie\widehat{G},\ \widehat{\mathfrak{h}}=\Lie\widehat{H}$. Then, $W_F\times \SL_2(\C)$ acts on $\widehat{\mathfrak{g}}$ via $\Ad\circ\varphi$. By definition of $\widehat{H}$, we have
    \[
        \widehat{\mathfrak{g}}^{I_F}=\widehat{\mathfrak{h}}.
    \]
    Since the L-function $L(s,\pi)$ of a Weil--Deligne representation $(\pi,V)$ is defined from the determinant of the geometric Frobenius on $V^{I_F}$, we have
    \[
        L(1-s,(\Ad\circ\varphi)^\vee)=L(1-s,(\Ad\circ\varphi_H)^\vee),\quad L(s,\Ad\circ\varphi)=L(s,\Ad\circ\varphi_H).
    \]
    Therefore 
    \[
        \frac{\gamma(s,\varphi,\Ad,\psi)}{\gamma(s,\varphi_H,\Ad,\psi)}=\frac{\epsilon(s,\Ad\circ\varphi,\psi)}{\epsilon(s,\Ad\circ\varphi_H,\psi)}.
    \]
    For a representation $\pi$ of $W_F\times \SL_2(\C)$, we consider the inclusion 
    \[
        W_F\hookrightarrow W_F\times\SL_2(\C);\ w\mapsto \left(w,\begin{bmatrix}
            q^{d(w)/2}&0\\
            0&q^{-d(w)/2}
        \end{bmatrix}\right),
    \]
    where $d\colon W_F\twoheadrightarrow \langle \Fr\rangle\cong \Z$, and denote by $\pi|_{W_F}$ the restriction of $\pi$ along it. Then, as in \cite[(8.12.4)]{Deligne}, $\epsilon(s,\Ad\circ\varphi,\psi)/\epsilon(s,\Ad\circ\varphi|_{W_F},\psi)$ is expressed in terms of $\widehat{\mathfrak{g}}^{I_F}=\widehat{\mathfrak{h}}$, so it does not change when $\varphi$ is replaced with $\varphi_H$. Thus we have
    \[
        \frac{\gamma(0,\varphi,\Ad,\psi)}{\gamma(0,\varphi_H,\Ad,\psi)}=\frac{\epsilon(0,\Ad\circ\varphi|_{W_F},\psi)}{\epsilon(0,\Ad\circ\varphi_H|_{W_F},\psi)}.
    \]
    Since $\varphi|_{P_F}$ and $\varphi|_{I_F}$ act trivially on $\widehat{\mathfrak{g}}$ and $\widehat{\mathfrak{h}}$, respectively, the right hand side of the above equation is $q^{(\dim\widehat{\mathfrak{g}}-\dim\widehat{\mathfrak{h}})/2}=q^{(\dim G'-\dim H')/2}$.

    Thus we have
    \[
        \frac{\fdeg\pi}{\gamma(0,\varphi,\Ad,\psi)}=\frac{\fdeg\pi_\unip}{\gamma(0,\varphi_H,\Ad,\psi)}.
    \]
    \cite[Theorem 3]{FOS} states that the right hand side is equal to $\dim(\epsilon)/\abs{\mathcal{S}^\natural_{\varphi_H}}$. Since the split center of $H$ is trivial, 
    \[
        \mathcal{S}^{\natural}_{\varphi_H}=\pi_0(Z_{\widehat{H}}(\varphi_H))=\pi_0(Z_{\widehat{G}}(\varphi))=\mathcal{S}^\natural_{\varphi}.
    \]
    Hence Conjecture~\ref{conj:HII} also holds for $\LLC$.
\end{proof}

\bibliography{reference}
\bibliographystyle{my_amsalpha}

\end{document}